\documentclass[a4paper,reqno, 11pt]{amsart}
\usepackage{amsmath}
\usepackage{paralist}
\usepackage{graphics}
\usepackage{epsfig}
\usepackage{amsfonts}
\usepackage{amssymb}
\usepackage{color}
\usepackage{epstopdf}
\usepackage[numbers,sort&compress]{natbib}

\newtheorem{theorem}{Theorem}[section]
\newtheorem{lemma}[theorem]{Lemma}

\def\ifl{\iffalse }

\def\bc{\begin{center}}       \def\ec{\end{center}}
\def\ba{\begin{array}}        \def\ea{\end{array}}
\def\be{\begin{equation}}     \def\ee{\end{equation}}
\def\bea{\begin{eqnarray}}    \def\eea{\end{eqnarray}}
\def\beaa{\begin{eqnarray*}}  \def\eeaa{\end{eqnarray*}}

\numberwithin{equation}{section}
\newtheorem{proposition}[theorem]{Proposition}

\newtheorem{remark}[theorem]{Remark}
\numberwithin{equation}{section}

\begin{document}

\title[Chemotaxis system, logistic source, and boundedness]
{How strong  a logistic damping can  prevent blow-up for the minimal  Keller-Segel chemotaxis system? }

\author{Tian Xiang}
\address{Institute for Mathematical Sciences, Renmin University of China, Beijing,  100872, China}
\email{txiang@ruc.edu.cn}

\subjclass[2000]{Primary:  35K59, 	35K51,  35K57, 92C17; Secondary: 35B44, 35A01.}


\keywords{Chemotaxis systems, logistic dampening,  boundedness,  global existence, blow-up, long time dynamics.}

\begin{abstract}
We study  nonnegative solutions  of  parabolic-parabolic Keller-Segel  minimal-chemotaxis-growth systems with prototype given by
$$\left\{ \begin{array}{lll}
&u_t = \nabla \cdot (d_1\nabla u-\chi u\nabla v)+\kappa u-\mu u^2,  &\quad x\in \Omega, t>0, \\[0.2cm]
& v_t =d_2 \Delta v  -\beta v+\alpha u,  &\quad x\in \Omega, t>0  \end{array}\right.
$$
 in a  smooth bounded smooth but not necessarily  \it{convex} domain  $\Omega\subset \mathbb{R}^n$ ($ n\geq3$) with nonnegative initial data $u_0, v_0$ and homogeneous Neumann boundary data, where $d_1,d_2,\alpha,\beta, \mu>0$, $\chi, \kappa\in \mathbb{R}$.

 We provide  quantitative and qualitative descriptions of the competition between logistic damping and other ingredients, especially,  chemotactic aggregation to guarantee boundedness and convergence. Specifically,  we first obtain an explicit  formula   $\mu_0=\mu_0(n, d_1,d_2, \alpha, \chi)$ for the logistic  damping rate $\mu$ such  that the system has no blow-ups whenever $\mu>\mu_0$. In particular,  for $\Omega\subset \mathbb{R}^3$, we get a clean formula for $\mu_0$:
$$
\mu_0(3, d_1,d_2,\alpha, \chi)= \left\{ \begin{array}{ll}
&\frac{3}{4d_1}\alpha\chi, \quad \text{if } d_1=d_2, \chi>0  \text{   and  } \Omega \text{  is convex},   \\[0.25cm]
&\frac{3}{\sqrt{10}-2}(\frac{1}{d_1}+\frac{2}{d_2})\alpha |\chi|,  \quad \quad  \text{  otherwise}. \end{array}\right.
$$
 This offers a quantized effect of the logistic source on the prevention of blow-ups. Our result  extends the fundamental  boundedness principle  by Winkler \cite{Win10} with $d_1=1, d_2=\alpha=\beta:=1/\tau$,  $\Omega$ being  convex and sufficiently large values of $\mu$ beyond a certain number not explicitly known (except the simple case $\tau=1$ and $\chi>0$) and quantizes the qualitative  result of Yang et al. \cite{YCJZ15}. Besides, in non-convex domains, since $\mu_0(3,1,1,1, \chi)=(7.743416\cdots) \chi.$
 the recent  boundedness result, $\mu> 20\chi$, of Mu and Lin \cite{ML16} is greatly improved.

 Then we derive another explicit formula:  $\mu_1=0$ for $\kappa\leq0$ and $$\mu_1=\mu_1(d_1,d_2,\alpha,\beta,\kappa,\chi)=\frac{\alpha|\chi|}{4}\sqrt{\frac{\kappa}{d_1d_2\beta}} \ \  \text{ for  } \kappa>0
  $$
  for the logistic damping rate  so that  convergence of bounded solutions is ensured and the  respective convergence rates are explicitly calculated out whenever $\mu>\mu_1$. Recent convergence results of He-Zheng \cite{HZ16}  are hence completed and refined.
 Furthermore, both $\mu_0(n, d_1,d_2,\chi, \alpha)$ and $\mu_1(d_1,d_2,\alpha,\beta,\kappa,\chi)$ tend to infinity   as $d_1\rightarrow 0$ or $d_2\rightarrow 0$ and they are decreasing in $d_1$ and $d_2$; therefore,  small diffusion, especially,  degenerate or nonlinear diffusion, enhances the possibility of the occurrence of "unbounded" solutions. This gives a clue on how to produce blow-up  solutions for Keller-Segel chemotaxis models with logistic source.
\end{abstract}

\maketitle

\section{Introduction and outline of the main results}

  Chemotaxis,  the tendency of  cells, bacteria and similarly tiny organisms  to orient  the direction of movement (otherwise random) toward increasing or decreasing  concentration of a signaling substance, has been attracting great attention  in biological and mathematical community.  A celebrated mathematical model, initially proposed by  Keller-Segel \cite{Ke, Ke1},  makes up of two parabolic equations reflecting chemotactic movement through a nonlinear advective-diffusive term as its most defining characteristic. Their pioneering works  have initiated vast
investigations  of  the K-S model  and its various forms of variants since 1970.  We  refer to  the beautiful survey papers  \cite{Ho1,Hi, Win13, BBTW15}, where a broad survey on the progress of various chemotaxis models  and  rich selection of references can be found.

If biological processes in which chemotaxis plays a role are modeled not only on small timescales,  often the spontaneous growth  of the population, whose density we will denote by $u$,  should be incorporated.  A prototypical choice to achieve this is the logistic type source $\kappa u-\mu u^2$ with birth and death rates $\kappa $ and $\mu$, respectively. Let us  then begin with the most simplest perhaps the most  interesting so-called minimal-chemotaxis-growth model
\begin{equation}\label{minimal-model}\left\{ \begin{array}{lll}
&u_t = \nabla \cdot ( \nabla u-\chi u \nabla v)+\kappa u-\mu u^2,  &\quad x\in \Omega, t>0, \\[0.2cm]
&\tau v_t = \Delta  v  - v+u,  &\quad x\in \Omega, t>0, \\[0.2cm]
&\frac{\partial u}{\partial \nu}=\frac{\partial v}{\partial \nu}=0, &\quad x\in \partial \Omega, t>0,\\[0.2cm]
&u(x,0)=u_0(x)\geq 0, v(x,0)=v_0(x)\geq 0, &\quad x\in \Omega,  \end{array}\right.  \end{equation}
where  $\Omega\subset \mathbb{R}^n, n\geq 1$ is a smooth bounded domain, $v$ denotes the concentration of the chemical signal,  $\kappa\geq 0, \mu\geq 0, \tau\geq 0$ and $\chi\in\mathbb{R}$. The nonlinear term  $\chi u\nabla v $,  the defining term in chemotaxis models, is called chemotactic term:   in the case $\chi>0$, it  models the cells movement toward the higher concentrations of the chemical signal, which is called (positive) chemotaxis, in the case $\chi<0$,   it models the cells movement away from  the higher concentrations of the chemical signal, which is called negative chemotaxis.

Model  \eqref{minimal-model} with $\kappa=0, \mu=0$  corresponds to the classical Keller-Segel minimal model \cite{Ke, Ke1}, which and whose variants have been extensively explored since 1970.  The striking  feature of KS type models is the possibility of blow-up of solutions in a finite/infinite time (see, e.g., \cite{Ho1, Nan, Win13} ), which strongly depends on the space dimension. A finite/infinite  time blow-up never occurs in $1$-dimension  \cite{OY01, HP04, Xiang14} (except in some extreme nonlinear denerate diffusion model \cite{CL10}),  a  critical  mass blow-up occurs in $2$-dimension: when the initial mass lies below the threshold solutions exist globally, while above the threshold solutions blow up in finite time \cite{NSY97, HW01, SS01}, and generic blow-up in $\geq 3$-D \cite{Win100, Win13}. The knowledge about the classical KS type models appears to be rather complete, see the aforementioned surveys for more.

The blow up  solution or a  $\delta$ function is surely connected to the
phenomenon of cell aggregation; on the other hand,  various mechanisms proposed to the model have manifested that  the blow-up solutions is fully precluded while  pattern formation arises \cite{CKWW12, Hi, MW-15-Non, Xiang15JDE}. Among those mechanisms (see the introduction in \cite{Xiang15JDE}),  inclusion of  a growth source of cells is a common choice. In particular, the presence of logistic source has been shown to have an effect  of preventing ultimate growth  of populations. Indeed, in the case $n=1,2$,  even arbitrarily  small $\mu>0$ will be sufficient to suppress blow-up by ensuring all solutions to  \eqref{minimal-model} are global-in-time and uniformly bounded for all reasonably initial data  \cite{OY01, OTYM02, HP04, Xiangpre}. This is even  true for a two-dimensional parabolic-elliptic chemotaxis system with singular sensitivity \cite{FW14}.  Whereas, in the case $n\geq 3$,  the first  boundedness and global existence were  obtained  for a parabolic-elliptic simplification of (\ref{minimal-model}), i.e.,  $\tau=0$,  under the condition that  $\mu>\frac{(n-2)}{n}\chi$ \cite{TW07}. Recently, this result was  improved to the borderline case $\mu\geq \frac{n-2}{n}\chi$ \cite{WXpre}.  See also the  existence of very weak solutions under more general conditions \cite{Win08}. For the full parabolic-parabolic minimal chemotaxis-growth mode,  fundamental findings  were  obtained  by Winkler \cite{Win10}. Under the additional assumptions that $\Omega$ is convex and  $\mu$ is beyond  a certain number $\mu_0$ not explicitly known (except the case $\tau=1$ and $\chi>0$, where $\mu>\frac{n}{4}\chi$ is sufficient to prevent blow-ups), he proved the existence and uniqueness of global, smooth, bounded solutions to \eqref{minimal-model}.  Recently, a progress on global boundedness to \eqref{minimal-model} with $\chi>0$  was derived as long as $\mu>\theta_0 \chi$ for some implicit positive constant  $\theta_0$ depending on Sobolev embedding constants \cite{YCJZ15}. In 2015, an explicit lower bound for a 3-D chemotaxis-fluid system with logistic source was obtained by Tao and Winkler \cite{TWZAMP}, when applied to the chemotaxis system \eqref{minimal-model} with $\chi=\tau=1$, their result states that $\mu\geq 23$ is enough to prevent blow-ups. This bound was further improved by  Mu and Lin  \cite{ML16} (2016) in three dimensions, wherein they replaced the logistic source in \eqref{minimal-model} by the damping term $u-\mu u^r$ with $r\geq 2$ to derive the boundedness  under
\be\label{ml-con}
\mu^{\frac{1}{r-1}} >20\chi.
\ee
Their arguments were done to the case that $\Omega$ is convex by remarking that they could be adapted to non-convex domains by virtue of the papers \cite{MS14,ISY14}. Of course, when $r>2$, this result was already implied by \cite{Win10,YCJZ15}. Moreover, for the particular choices  $\kappa=\tau=1$, under certain largeness condition on the ration  $\mu/\chi$, the stabilization of bounded solution $(u,v)$ of \eqref{minimal-model} to the constant equilibrium $(1/\mu, 1/\mu)$ as $t\rightarrow \infty$ occurs \cite{Win14, HZ16, ML16}. While,  for arbitrarily small $\mu>0$,  only existence of global weak solutions to \eqref{minimal-model} is available \cite{La15-JDE} in convex 3-D domains.  Other  dynamical properties of \eqref{minimal-model} can be found in \cite{FLP07, HP11, Win14-JNS,  Cao15, La15-DCDS, TW15-JDE}.  Finally,  we observe that enormous  variants of (\ref{minimal-model}) have been considered  to provide conditions on diffusion, degradation, chemo-sensitivity and mostly on the growth source ensuring the boundedness of the proposed models \cite{ TW12, BH13, NO13,  Cao14, LX14, WLM14, WMZ14,Xiangpre, ZL15-ZAMP, Zh15-JMAA} and the references therein, and  that explosion of solutions  is possible  in chemotaxis systems  despite logistic growth restriction \cite{Win11, ZMH-15}. Therefore, it is meaningful to detect more circumstances  where no blow-up is allowed for the minimal KS model with logistic growth  \eqref{minimal-model}.

It is widely known that the KS minimal model \eqref{minimal-model} with $\kappa=\mu=0$  admits both bounded and unbounded solutions, identified via the critical chemotactic sensitivity  $u^{\frac{2}{n}}$ \cite{HW05, MA13, CS15-JDE}. Therefore, the model \eqref{minimal-model}  is simply a supercritical case with the balance of logistic damping and aggregation effects, for which the property of solutions should be not only qualitatively but also quantitatively determined by the parameters involved.  Motivated by the works \cite{Win10, YCJZ15, TWZAMP, ML16}, we attempt to provide a quantitative description of the competition between  logistic damping and other ingredients, especially,  chemotactic aggregation, and, in particular, we aim to find a full picture on how the lower bound $\mu_0$ of the logistic damping rate $\mu$ is affected by all the involving parameters so that no blow-up is allowed for $\mu>\mu_0$. Therefore, in this paper,  we will consider a full parameter K-S minimal system with a growth source covering the standard logistics source as follows:
\begin{equation}\label{para-para}\left\{ \begin{array}{lll}
&u_t = \nabla \cdot (d_1 \nabla u-\chi u \nabla v)+f(u),  &\quad x\in \Omega, t>0, \\[0.2cm]
&v_t = d_2\Delta  v  -\beta v+\alpha u,  &\quad x\in \Omega, t>0, \\[0.2cm]
&\frac{\partial u}{\partial \nu}=\frac{\partial v}{\partial \nu}=0, &\quad x\in \partial \Omega, t>0,\\[0.2cm]
&u(x,0)=u_0(x)\geq 0, v(x,0)=v_0(x)\geq 0, &\quad x\in \Omega,  \end{array}\right.  \end{equation}
where   $\Omega\subset \mathbb{R}^n, n\geq 1$ is a smooth bounded domain but not necessarily convex,  $d_1,d_2,\alpha,\beta>0$, $\chi\in\mathbb{R}$ and $f:\mathbb{R}\rightarrow \mathbb{R}$ is smooth and satisfies $f(0) \geq 0$ as well as
 \begin{equation}\label{log-con}
f(s)\leq a-\mu s^2, \quad \forall s\geq 0
\end{equation}
for some $a\geq 0$ and $\mu>0$.

With the aid of  the boundedness criteria on how the growth source affects the boundedness for a general class of chemotaxis-growth systems than \eqref{para-para} \cite{BBTW15, Xiangpre}, we provide a detailed algorithm to derive  an explicit  formula  for the lower bound $\mu_0$ of the logistic  damping rate $\mu$ such  that the system \eqref{para-para} admits only globally bounded solutions whenever $\mu>\mu_0$.  For nonconvex domains, our procedure is mainly carried out in physically relevant  settings ($n=3,4, 5$), where we have a clean and compact formula for $\mu_0$. Precisely, our main quantitative findings in this regard read as follows:
 \begin{theorem}\label{main thm}[How strong  a logistic damping can  prevent blow-up for  \eqref{para-para}]  Let $\Omega\subset \mathbb{R}^ n(n\geq 3)$ be a bounded smooth domain,  $f$ satisfy  the logistic condition \eqref{log-con} and $d_1,d_2,\alpha,\beta>0$, $a\geq 0$ and $\chi\in\mathbb{R}$.
\
\
\
\begin{itemize}
\item[(i)]  For $n=3$,  let the lower logistic damping rate  $\mu_0=\mu_0(3, d_1,d_2,\alpha,\chi)$ be explicitly given by
\begin{equation}\label{mu0-con-3d}
\mu_0= \left\{ \begin{array}{ll}
&\frac{3}{4d_1}\alpha\chi, \quad \quad \quad \text{if } d_1=d_2, \chi>0  \text{   and  } \Omega \text{  is convex},   \\[0.25cm]
&\frac{3}{\sqrt{10}-2}(\frac{1}{d_1}+\frac{2}{d_2})\alpha |\chi|,  \text{  otherwise};  \end{array}\right.
\end{equation}
\item[(ii)] for $n=4, 5$, let the lower logistic damping rate $\mu_0=\mu_0(n,d_1,d_2,\alpha,\chi)$ be  explicitly given by
\begin{equation}\label{mu0-con-4d}
\mu_0= \left\{ \begin{array}{ll}
&\frac{n}{4d_1}\alpha\chi, \quad \quad \quad \text{if } d_1=d_2, \chi>0  \text{   and  } \Omega \text{  is convex },   \\[0.25cm]
&\max\Bigr\{\frac{1}{3}h(n,d_1,d_2), \frac{n}{\sqrt{2n+4}-2}(\frac{1}{d_1}+\frac{2}{d_2})\Bigr\}\alpha |\chi|,  \text{  otherwise } \end{array}\right.
\end{equation}
with
$$
\ba{ll}
h(n,d_1,d_2)=\min_{0<\epsilon<d_1, 0<\eta<d_2}&\Big\{\sqrt{\frac{n}{18d_2\epsilon}}+\sqrt{\frac{1}{2\epsilon}(\frac{1}{\eta}+\frac{n}{2d_2})}\\[0.25cm]
&+\sqrt{\frac{1}{(d_2-\eta)}(\frac{2}{\eta}+\frac{n}{2d_2})}\Bigr[\sqrt{2}+\frac{(d_1+d_2)}{2\sqrt{(d_1-\epsilon)(d_2-\eta)}}\Bigr]\Bigr\}.
\ea
$$
\end{itemize}
Then, whenever $\mu>\mu_0$,  the chemotaxis-growth system \eqref{para-para} has a unique global solution $(u,v)$ for which both $u$ and $v$ are nonnegative and bounded in $\Omega\times (0,\infty)$.
\end{theorem}
\begin{remark}\label{note on main thm} [Notes on  how strong a logistic damping can prevent blow-up]
\
\
\
\begin{itemize}
\item[(P1)] The explicit logistic damping rate $\mu_0$ given in \eqref{mu0-con-3d} or  \eqref{mu0-con-4d} exhibits the contributions of  the degradation, creation   and diffusion rates, etc in respective of boundedness of solutions of \eqref{para-para}. That is, it  shows how  strong  a logistic damping is needed to prevent blow-ups for  \eqref{para-para}.
\item [(P2)] The formula for $\mu_0$ is not  only explicitly expressible but also is independent of the degradation rate $\beta$ of signals,   the birth rate $a$ of cells,   the size of  domain $\Omega$,  initial data $u_0,v_0$ and Sobolev embedding constants. This gives a quantized effect of the logistic source on preventing blow-ups, and hence improves the boundedness principles  \cite{Win10, Xiangpre} and the qualitative result \cite{YCJZ15} .
\item[(P3)] The form of $\mu_0(n, d_1,d_2, \alpha, \chi)$ captures our common understanding, since either  $\chi=0$ (no chemotaxis, cf. \cite[Proposition 2.6] {Xiangpre})  or $\alpha=0$ (decoupled),  the boundedness and global existence are easily seen for any $\mu>0$.

\item[(P4)] The  chemo-repulsion case, i.e., $\chi<0$ is allowed as well in $\mu_0$.
\item [(P5)] $\mu_0(n, d_1,d_2, \alpha, \chi)\rightarrow \infty$ as $d_1\rightarrow 0$ or $d_2\rightarrow 0$ (and it  is decreasing in $d_1$ and $d_2$); thus small diffusion, especially, degenerate or nonlinear diffusion, enhances the possibility of the occurrence of blow-ups.

\item[(P6)]  In the case  that  $\Omega$ is nonconvex, we have
$$
\mu_0^{(\text{nonconvex})}(3, 1,1,1,\chi)=\frac{9}{\sqrt{10}-2}\chi=(7.743416\cdots)\chi.
$$
Hence,  the very recent  boundedness result,  $\mu>20\chi$ obtained from \eqref{ml-con} with $r=2$, of Mu and Lin \cite{ML16}, and the byproduct boundedness,  $\mu\geq \mu_0(3, 1,1,1, 1)=23$ of  Tao and Winkler \cite{TWZAMP}, as  their studied 3-D fluid system coupled with the minimal chemotaxis system \eqref{minimal-model}, are greatly improved.
\end{itemize}
\end{remark}

For mathematical completeness, one may wonder such an explicit formula is also available in $\geq 6$ dimensions. Indeed, our algorithm suggests that an explicit formula for $\mu_0$ in $n$-D ($n\geq  6$) of the form $\theta_0(n,d_1,d_2)\alpha |\chi|$,  which is not clean but enjoys the first 5 properties (P1)-(P5), would be  also available:
\begin{itemize}
\item[(Q1)] For general $n\geq 6$,  no blow-ups can occur to the minimal-chemotaxis-growth model \eqref{para-para} if $\mu> \mu_0(n, d_1,d_2,\alpha,\chi)$, where
\begin{equation}\label{mu0-con-general}
 \mu_0(n, d_1,d_2,\alpha,\chi)= \left\{ \begin{array}{ll}
&\frac{n}{4d_1}\alpha\chi, \quad \quad \quad \text{if } d_1=d_2, \chi>0  \text{   and  } \Omega \text{  is convex},   \\[0.25cm]
&\theta_0(n,d_1,d_2)\alpha |\chi|,  \text{  otherwise } \end{array}\right.
\end{equation}
with some  explicit  (perhaps messy) formula  $\theta_0$ depending only on $n,d_1$ and $d_2$ with the property that $\theta_0\rightarrow \infty$ as either $d_1\rightarrow 0$ or $d_2\rightarrow 0$.
\end{itemize}
Indeed, the first case of \eqref{mu0-con-general} has been shown in \cite{Win10}  for the KS model \eqref{minimal-model} with $\tau=1$,  see  Lemma  \ref{main-thm-nd-convex}  for the full-parameter model \eqref{para-para}. Based on \cite{YCJZ15} and carefully  and painfully scrutinizing our procedure, it is quite possible to trace out the formula $\theta_0$ in \eqref{mu0-con-general}. Here, we  leave the rigorous justification for future investigations.

In nonconvex domains, the formula $\mu_0(3, d_1,d_2,\alpha,\chi)$ is the smallest damping rate that we could obtain using this procedure. While, in convex domains, we have  $\mu_0^{(\text{convex})}(3, 1,1,1,1)=0.75$. A comparison to  $\mu_0^{(\text{nonconvex})}(3, 1,1,1,1)=7.743416\cdots$ as computed in (P3),  indicates that convexity makes a big difference and that  the formula $\mu_0(n, d_1,d_2,\alpha,\chi)$  may not be optimal even through it meets the expected properties described in (P2) and (P3), on the other hand.

This discussions lead us to ask other challenging questions left for the minimal-chemotaxis-growth model \eqref{para-para}:
\begin{itemize}
\item[(Q2)]  Is there a critical damping rate $\mu_0^c$ that distinguishes between occurrence and impossibility of blow-up for \eqref{para-para}. That is, when the logistic  source satisfies $f(u) \leq a-\mu u^2$ with $\mu<\mu_0^c$, blow-up occurs; when the source satisfies $f(u) \leq a-\mu u^2$ with $\mu>\mu_0^c$, blow-up is impossible.
\item[(Q3)] What happens for small logistic damping $\mu<\mu_0$, boundedness or blow-up?
\end{itemize}

Questions akin to (Q2) and (Q3)  may have been indicated by existing literature \cite{Win10, Xiangpre}.  To explore them,  a combination of  the references \cite{TW07, Win11, STW14, Win14-JNS, Win14,  La15-DCDS, TW15-JDE, TW15-pre, WXpre} may be of some help. A complete  (quantitative or qualitative) description of the competition between chemotactic aggregation and logistic damping is definitely worthwhile for future explorations.

  When logistic damping wins over chemotactic aggregation, i.e.,  $\mu>\mu_0$ or equivalently $|\chi|<\frac{\mu}{\alpha \theta_0(n, d_1,d_2)}$, we wish to see  the explicit effects  of  each term in \eqref{para-para} on the long time dynamical properties of bounded solutions. To this end, we study the large time behavior for the minimal chemotaxis  model with a standard logistic source:
\begin{equation}\label{min-std}\left\{ \begin{array}{lll}
&u_t = \nabla \cdot (d_1 \nabla u-\chi u \nabla v)+\kappa u-\mu u^2,  &\quad x\in \Omega, t>0, \\[0.2cm]
&v_t = d_2\Delta  v  -\beta v+\alpha u,  &\quad x\in \Omega, t>0, \\[0.2cm]
&\frac{\partial u}{\partial \nu}=\frac{\partial v}{\partial \nu}=0, &\quad x\in \partial \Omega, t>0,\\[0.2cm]
&u(x,0)=u_0(x)\geq 0, v(x,0)=v_0(x)\geq 0, &\quad x\in \Omega.   \end{array}\right.  \end{equation}
 For $d_1=d_2=\alpha=\beta=\kappa=1$, the uniform convergence of bounded solution to $(1/\mu, 1/\mu)$ is first proved in \cite[Theorem 1.1]{Win14} for  $\mu/\chi>0$ sufficiently large and then in  \cite[Theorem 1.3]{ML16} for $\mu>20\chi$;  for $d_1=d_2=\alpha=\beta=1$, He and Zheng \cite{HZ16} modified the energy functional method from \cite{BW16,Hsu78} to obtain the stabilities  of the constant equilibria $(0,0)$ and $(\kappa/\mu, \kappa/\mu)$ with  convergence estimates \cite[Theorem 3]{HZ16}. Next, for completeness and to see the role of other parameters in the large time behavior of the solutions, we combine the energy functional method from \cite{BW16,HZ16, Hsu78} to show the stability of the constant equilibria $(0,0)$ and $(\kappa /\mu, \alpha \kappa /(\beta\mu))$ for the full-parameter KS model \eqref{min-std}. Our precise results on the large time limit of bounded solutions of \eqref{min-std} are collected in the following theorem.
 \begin{theorem}\label{large time} Let $\Omega\subset \mathbb{R}^ n(n\geq 3)$ be a bounded smooth domain, $u_0\geq, \not\equiv 0$,  $d_1,d_2,\alpha,\beta, \mu>0$, $\kappa, \chi\in \mathbb{R}$ and, finally,  let $\mu>\mu_0$ as obtained in Theorem \ref{main thm} so that $(u,v)$ be a global and bounded solution of \eqref{min-std}.
\begin{itemize}
\item[(i)] When $\kappa>0$, assume additionally that
\be\label{mu-con} \mu>\mu_1(d_1,d_2,\alpha,\beta,\kappa,\chi)=\frac{\alpha|\chi|}{4}\sqrt{\frac{\kappa}{d_1d_2\beta}}.  \ee
Then the  solution $(u,v)$ of \eqref{min-std} converges exponentially:
\be\label{u+v-conv-i}
\|u(\cdot, t)-\frac{\kappa}{\mu}\|_{L^\infty(\Omega)}+\|v(\cdot, t)-\frac{\alpha\kappa}{\beta\mu}\ \|_{L^\infty(\Omega)}\leq C
 e^{-\gamma t}
 \ee
for all $t\geq 0$ and some large constant $C$ independent of $t$ and
$$
\gamma=\frac{\min\Bigr\{(\mu-\frac{\alpha \kappa\chi^2}{4d_1d_2\mu}\epsilon_0),\frac{\kappa\chi^2}{4d_1d_2\mu}(\beta-\frac{\alpha}{4\epsilon_0})\Bigr\}}
{(n+2)\max\{\frac{\mu}{\kappa},\frac{\kappa\chi^2}{8d_1d_2\mu}\}}, \quad \epsilon_0=\frac{1}{2}(\frac{\alpha}{4\beta}+\frac{4d_1d_2\mu^2}{\alpha\kappa \chi^2}).
$$
\item[(ii)] When $\kappa=0$, the  solution $(u,v)$ of \eqref{min-std} converges algebraically:
    \be\label{u+v-conv-ii}
\|u(\cdot,t)\|_{L^\infty(\Omega)}+\|v(\cdot,t)\|_{L^\infty(\Omega)}
 \leq C (t+1)^{-\frac{1}{n+1}}
 \ee
for all $t\geq 0$ and some large constant $C$ independent of $t$.
\item[(iii)] When  $\kappa<0$, the solution $(u,v)$ of \eqref{min-std} converges exponentially:
\be\label{u+v-conv-iii}
\|u(\cdot,t)\|_{L^\infty(\Omega)}
\leq C e^{\frac{\kappa}{n+1}t },\quad
\|v(\cdot,t)\|_{L^\infty(\Omega)}\leq C e^{-\frac{1}{2(n+1)}\min\{\beta, -\kappa\}t}
\ee
for all $t\geq 0$ and some large constant $C$ independent of $t$.
\end{itemize}
 \end{theorem}
  \begin{remark}\label{note on lt thm} [Explicit effects  on convergence and  convergence rate]
\
\
\
\begin{itemize}
\item[(P7)] The formula  $\mu_1(d_1,d_2,\alpha,\beta,\kappa,\chi)$ exhibits the explicit effects of each gradient in \eqref{para-para} on convergence and it enjoys the property described in (P3). Besides, our convergence completes and refines \cite[Theorem 3 for $d_1=d_2=\alpha=\beta=1$ ]{HZ16}  by computing out the explicit rates of convergence.
\item [(P8)] From Remark \ref{note on main thm}, we find that $\kappa$ and $\beta$ do not play any role in boundedness, while, they play big roles in the long time behavior as seen in \eqref{mu-con}. These effects have not been detected yet in the literature. In particular, $\mu_1( d_1,d_2, \alpha, \beta, \kappa, \chi)$ is decreasing in $d_1$, $d_2$ and $\beta$, and $\mu_1( d_1,d_2, \alpha, \beta, \kappa, \chi)\rightarrow \infty$ as $d_1\rightarrow 0$ or $d_2\rightarrow 0$ or $\beta\rightarrow 0$. Therefore,  small diffusion, especially, degenerate or small degradation,  makes the stabilization harder.
    \end{itemize}
\end{remark}

    In chemotaxis-growth systems,  the most challenging and interesting wide-open question is to detect the possibility of finite/infinite-time physical blow-ups ($n=3$) (nonphysical blow-ups has been demonstrated in \cite{Win11, ZMH-15} for $n\geq 5$).   Remarks \ref{note on main thm}  and \ref{note on lt thm}, especially,  (P5) and (P8)  suggest certain clues on how to produce unbounded solutions \cite{BCL-09, Win10-2, WD10, CS12JDE, CS15-JDE, PN-16-Poin}. To attack such a challenging problem,  one may try the following chemotaxis  system and perhaps its simplified version:
\be \label{min-ks-bp}\left\{ \begin{array}{lll}
&u_t = \nabla \cdot (\epsilon_1(u) \nabla u-  u \nabla v)+  \kappa u-\mu u^2,  &\quad x\in \Omega, t>0, \\[0.2cm]
&\tau v_t = \epsilon_2\Delta  v  -\epsilon_3 v+ u,  &\quad x\in \Omega, t>0
   \end{array}\right.  \ee
with $\epsilon_i>0(i=2,3)$ being sufficiently small and either  $\epsilon_1(u)$ being a sufficiently small positive constant or $\epsilon_1(u)\rightarrow 0$ as $u\rightarrow \infty$ (very slow diffusion at point of high densities).
Indeed,  for $\tau=0, \epsilon_1(u)=0$, $\epsilon_2=\epsilon_3=1$ in \eqref{min-ks-bp},  verifications can be found in \cite{Win14-JNS, La15-DCDS}. We leave the challenging exploration of the possibility of  blow-ups to \eqref{min-ks-bp} (and hence (Q2) and (Q3)) for our future studies.

\section{Preliminaries and  subtle inequalities for \eqref{minimal-model}}

For convenience, we start with Young's inequality, which states, for any positive numbers $p$ and $q$ with $\frac{1}{p}+\frac{1}{q}=1$, that
$$
ab\leq \frac{a^p}{p}+\frac{b^q}{q}, \quad \quad \forall a,b\geq 0.
$$
This immediately implies the so-called Young's inequality with $\epsilon$:
\begin{lemma}\label{Y-epsilon}(Young' s inequality with $\epsilon$) Let $p$  and $q$ be two given positive numbers with $\frac{1}{p}+\frac{1}{q}=1$. Then, for any  $\epsilon>0$, it holds
$$
ab\leq \epsilon a^p+\frac{b^q}{(\epsilon p)^{\frac{q}{p}}q},  \quad \quad \forall a,b\geq 0.
$$\end{lemma}

 The  local  solvability and extendibility   of  the parabolic-parabolic chemotaxis system (\ref{para-para}) is well-established by using a suitable fixed point argument and standard parabolic regularity theory; see, for example, \cite{HW05, Win10, TW12, LX14}.
\begin{lemma}\label{local-in-time}Let  $d_1,d_2,\alpha,\beta>0$, $a\geq 0$, $\chi\in \mathbb{R}$  and let  $\Omega\subset \mathbb{R}^n$ be a bounded domain with a smooth boundary. Suppose  that the initial data $(u_0,v_0)$ satisfies $u_0\in C(\overline{\Omega})$  and $v_0\in W^{1,p}(\Omega)$ with some $p>n$ and that $f\in W^{1,\infty}_{\mbox{loc}}(\mathbb{R})$ with $f(0)\geq 0$. Then there is a unique,  nonnegative,  classical maximal solution $(u,v)$ of the IBVP (\ref{para-para}) on some maximal interval $[0, T_{\mbox{max}})$ with $0<T_{\mbox{max}} \leq \infty$ such that
$$\ba{ll}
u\in C(\overline{\Omega}\times [0, T_{\mbox{max}}))\cap C^{2,1}(\overline{\Omega}\times (0, T_{\mbox{max}})), \\[0.2cm]
v\in C(\overline{\Omega}\times [0, T_{\mbox{max}}))\cap C^{2,1}(\overline{\Omega}\times (0, T_{\mbox{max}}))\cap L_{\text{loc}}^\infty([0, T_{\mbox{max}}); W^{1,p}(\Omega)).
\ea $$
In particular,  if $T_{\mbox{max}}<\infty$, then
\be\label{T_m-cri} \|u(\cdot, t)\|_{L^\infty(\Omega)}+\|v(\cdot, t)\|_{W^{1,p}(\Omega)}\rightarrow \infty \quad \quad \mbox{ as }  t\rightarrow T_{\mbox{max}}^{-}.
\ee
\end{lemma}
For chemotaxis model without growth source, the total masses of cells are conserved.  While, for chemotaxis system with logistic growth, this is not true but the $L^1$-norm is bounded by integrating the $u$-equation in \eqref{para-para}. The following basic lemma has been well-known, c.f. \cite{Win10, Xiangpre} for instance.
\begin{lemma}\label{ul1-vgradl2} Let $f$ satisfy the logistic condition \eqref{log-con}. Then  the solution $(u,v)$ of \eqref{para-para} satisfies
$$
\int_\Omega u\leq \|u_0\|_{L^1}+(a+\frac{1}{4\mu})|\Omega|
$$
and
$$
  \int_\Omega |\nabla v|^2\leq \|\nabla v_0\|_{L^2}^2 +\frac{2\alpha^2}{bd_2}\Bigr[\frac{a}{2\beta}|\Omega|+\|u_0\|_{L^1}+(a+\frac{1}{\mu})|\Omega|\Bigr]
$$
 for all   $t\in  [0, T_{\text{max}}) $.
\end{lemma}
Next, we establish  three subtle  commonly used lemmas, which are  refined results of the corresponding  \cite[Lemmas 2.2-2.4]{Win10} even for $\Omega$ being convex. In such case, the ideas used to derive these inequalities are known, cf. \cite{Win10, TW12} for example.
\begin{lemma}  \label{u-lp-lem-optimal} Let $f$ satisfy the logistic condition \eqref{log-con}. Then, for any $p\geq1$, the solution $(u,v)$ of \eqref{para-para} satisfies
\begin{equation}\label{ulp: est}
\begin{split}
\frac{1}{p}\frac{d}{dt}\int_\Omega u^p&+(p-1)(d_1-\epsilon)\int_\Omega u^{p-2} |\bigtriangledown u|^2+\mu\int_\Omega u^{p+1}\\
&\leq \frac{ (p-1)\chi^2}{4\epsilon} \int_\Omega u^p|\bigtriangledown v|^2+a\int_\Omega u^{p-1}
\end{split}
\end{equation}
for all $t\in(0, T_{\text{max}})$ and for any $\epsilon\in(0, d_1)$.
\end{lemma}
\begin{proof}For  any $p\geq1$, multiplying the $u$-equation in (\ref{para-para}) by $u^{p-1}$ and integrating over $\Omega$ by parts, using  Young's inequality with $\epsilon$ and the logistic condition \eqref{log-con},   we conclude  that
\begin{align*}
&\frac{1}{p}\frac{d}{dt} \int_\Omega u^p+(p-1)d_1\int_\Omega  u^{p-2} |\bigtriangledown u|^2\\
&=(p-1)\chi \int_\Omega u^{p-1}\bigtriangledown u\bigtriangledown v+\int_\Omega f(u)u^{p-1}\\
& \leq (p-1)\epsilon\int_\Omega  u^{p-2} |\bigtriangledown u|^2+\frac{ (p-1)\chi^2}{4\epsilon} \int_\Omega u^p|\bigtriangledown v|^2+\int_\Omega u^{p-1}(a-\mu u^2), \end{align*}
which gives the desired inequality \eqref{ulp: est}.
\end{proof}

\begin{lemma}  \label{gradv-2q-lem-optimal} For any $q\geq1$, the solution $(u,v)$ of \eqref{para-para} satisfies
\begin{align}\label{v-grdv-2q}
&\frac{1}{q}\frac{d}{dt} \int_\Omega |\nabla v|^{2q}+(q-1)(d_2-\eta)\int_\Omega  |\nabla v|^{2(q-2)}|\nabla|\nabla v|^2|^2+2\beta\int_\Omega |\nabla v|^{2q}\nonumber\\
&\leq \Bigr[\frac{ (q-1)\alpha^2}{\eta}+\frac{n\alpha^2}{2d_2}\Bigr]\int_\Omega u^2 |\nabla v|^{2(q-1)}+d_2\int_{\partial\Omega}  |\nabla v|^{2(q-1)}\frac{\partial}{\partial \nu} |\nabla v|^2 \end{align}
for all $t\in(0, T_{\text{max}})$ and for any $\eta\in(0, d_2)$.
\end{lemma}
\begin{proof} For  any $q\geq 1$, by using the  $v$-equation in \eqref{para-para} and integrating by parts, we deduce  that
\begin{align*}
&\frac{1}{q}\frac{d}{dt} \int_\Omega |\nabla v|^{2q}= 2\int_\Omega |\nabla v|^{2(q-1)}\nabla v\cdot(d_2\nabla \Delta v-\beta\nabla  v+\alpha\nabla u)\\
&=d_2\int_\Omega |\nabla v|^{2(q-1)}\Delta |\nabla v|^2 -2d_2\int_\Omega  |\nabla v|^{2(q-1)}|D^2v|^2-2\beta\int_\Omega |\nabla v|^{2q}\\
&\quad +2\alpha\int_\Omega \nabla u\nabla v|\nabla v|^{2(q-1)}  \\
&=-(q-1)d_2\int_\Omega |\nabla v|^{2(q-2)}|\nabla|\nabla v|^2|^2+d_2\int_{\partial\Omega}  |\nabla v|^{2(q-1)}\frac{\partial}{\partial \nu} |\nabla v|^2 \\
&\quad -2d_2\int_\Omega  |\nabla v|^{2(q-1)}|D^2v|^2-2\beta\int_\Omega |\nabla v|^{2q}+2\alpha\int_\Omega \nabla u\nabla v|\nabla v|^{2(q-1)},  \end{align*}
where, from the first to the second line,  we have used the point-wise identity
\be\label{point-wise-id}
2\bigtriangledown v\cdot \bigtriangledown \bigtriangleup v=\bigtriangleup|\bigtriangledown v|^2-2|D^2v|^2, \quad |D^2v|^2=\sum_{i,j=1}^n|v_{x_ix_j}|^2.
\ee
Now, we utilize  integration by parts,  Young's inequality with $\epsilon$  and the fact that
\be\label{C-S-ineq-c}
|\bigtriangleup v|^2=\Bigr(\sum_{i=1}^nv_{x_ix_i}\Bigr)^2\leq n\sum_{i=1}^n|v_{x_ix_i}|^2\leq n|D^2v|^2 \ee
  to estimate the last integral as follows:
\begin{align*}
&2\alpha\int_\Omega \nabla u\nabla v|\nabla v|^{2(q-1)}\\
&=-2 (q-1)\alpha\int_\Omega u |\nabla v|^{2(q-2)}\nabla v\cdot \nabla |\nabla v|^2-2\alpha \int_\Omega u|\nabla v|^{2(q-1)}\Delta v\\
&\leq (q-1)\eta \int_\Omega |\nabla v|^{2(q-2)}| \nabla |\nabla v|^2|^2+\frac{ (q-1)\alpha^2}{\eta}\int_\Omega u^2 |\nabla v|^{2(q-1)}\\
&\quad +2d_2\int_\Omega  |\nabla v|^{2(q-1)}|D^2v|^2+\frac{n\alpha^2}{2d_2}\int_\Omega u^2 |\nabla v|^{2(q-1)}. \end{align*}
Combining these two results, we obtain the desired inequality \eqref{v-grdv-2q}.
\end{proof}
\begin{lemma}  \label{key-lem-optimal} Let $f$ satisfy the logistic condition \eqref{log-con}. Then, for any $p\geq 1, q\geq1$, the solution $(u,v)$ of \eqref{para-para} verifies
\begin{align}
&\frac{d}{dt} \int_\Omega u^p|\nabla v|^{2q}+(p-1)(d_1-\epsilon)p\int_\Omega u^{p-2}|\nabla u|^2|\nabla v|^{2q}+2\beta q\int_\Omega u^p |\nabla v|^{2q}\nonumber\\
&+(q-1)(d_2-\eta)q\int_\Omega u^{p}|\nabla v|^{2(q-2)}|\nabla |\nabla v|^2|^2+\mu p\int_\Omega u^{p+1} |\nabla v|^{2q}\nonumber\\
&\leq\frac{\chi^2 (p-1)p}{4\epsilon}\int_\Omega u^p|\nabla v|^{2(q+1)}+\chi pq\int_\Omega u^{p}|\nabla v|^{2(q-1)}\nabla v\nabla |\nabla v|^2 \nonumber\\
& \quad -(d_1+d_2)pq\int_\Omega u^{p-1}|\nabla v|^{2(q-1)}\nabla u\nabla |\nabla v|^2\nonumber\\
&\quad +\frac{q\alpha^2}{(p+1)^2}\Bigr[\frac{(q-1)}{\eta}+\frac{n}{2d_2}\Bigr]\int_\Omega u^{p+2}|\nabla v|^{2(q-1)}\nonumber\\
 &\quad +ap\int_\Omega u^{p-1} |\nabla v|^{2q}+d_2q\int_{\partial \Omega}u^p|\nabla v|^{2(q-1)}\frac{\partial |\nabla v|^2}{\partial \nu}\label{up-grdv2q-f}\end{align}
for all $t\in(0, T_{\text{max}})$ and for any $\epsilon\in (0, d_1)$ and  $\eta\in(0, d_2)$.
\end{lemma}
\begin{proof}
We use the $u$ and $v$-equations in \eqref{para-para}, \eqref{point-wise-id}, no flux boundary conditions for $u$ and $v$ and integration by parts to compute honestly that
\begin{align}
&\frac{d}{dt} \int_\Omega u^p|\nabla v|^{2q}\nonumber\\
&= p\int_\Omega u^{p-1}|\nabla v|^{2q}[\nabla \cdot(d_1\nabla u-\chi u\nabla v)+f(u)]\nonumber\\
&\quad +2q\int_\Omega u^p|\nabla v|^{2(q-1)}\nabla v\cdot(d_2\nabla \Delta v-\beta\nabla  v+\alpha\nabla u)\nonumber\\
&=-p\int_\Omega (d_1\nabla u-\chi u\nabla v)\Bigr[(p-1)u^{p-2}\nabla u |\nabla v|^{2q}+qu^{p-1}|\nabla v|^{2(q-1)}\nabla|\nabla v|^2\Bigr]\nonumber\\
&\quad+p\int_\Omega u^{p-1}|\nabla v|^{2q}f(u)+d_2q\int_\Omega u^p|\nabla v|^{2(q-1)}\Delta |\nabla v|^2-2\beta q\int_\Omega u^p |\nabla v|^{2q} \nonumber\\
&\quad-2d_2q\int_\Omega u^p |\nabla v|^{2(q-1)}|D^2v|^2+2\alpha q\int_\Omega u^p|\nabla v|^{2(q-1)}\nabla u\nabla v \nonumber \\
&=-(d_1+d_2)pq\int_\Omega u^{p-1}|\nabla v|^{2(q-1)}\nabla u\nabla |\nabla v|^2-d_1(p-1)p\int_\Omega u^{p-2}|\nabla u|^2|\nabla v|^{2q}\nonumber\\
&\quad+p\int_\Omega u^{p-1}|\nabla v|^{2q}f(u)+\chi (p-1)p\int_\Omega u^{p-1}|\nabla v|^{2q}\nabla u\nabla v\nonumber\\
&\quad+\chi pq\int_\Omega u^{p}|\nabla v|^{2(q-1)}\nabla v\nabla |\nabla v|^2-d_2q(q-1)\int_\Omega u^{p}|\nabla v|^{2(q-2)}|\nabla |\nabla v|^2|^2\nonumber\\ &\quad-2d_2q\int_\Omega u^p |\nabla v|^{2(q-1)}|D^2v|^2-2\beta q\int_\Omega u^p |\nabla v|^{2q}\nonumber\\
&\quad+2\alpha q\int_\Omega u^p|\nabla v|^{2(q-1)}\nabla u\nabla v +d_2q\int_{\partial \Omega}u^p|\nabla v|^{2(q-1)}\frac{\partial |\nabla v|^2}{\partial \nu}.\label{up-grdv2q-1}\end{align}
The logistic condition $f(u)\leq a-\mu u^2$ gives rise to
\be\label{log-up-est}
p\int_\Omega u^{p-1}|\nabla v|^{2q}f(u)\leq ap\int_\Omega  u^{p-1}|\nabla v|^{2q}-\mu p\int_\Omega  u^{p+1}|\nabla v|^{2q}.
\ee
A simple use of  Young's inequality with $\epsilon$ shows
\begin{align}
&\chi (p-1)p\int_\Omega u^{p-1}|\nabla v|^{2q}\nabla u\nabla v\nonumber\\
&\leq \epsilon (p-1)p \int_\Omega u^{p-2}|\nabla u|^2|\nabla v|^{2q}+\frac{\chi^2 (p-1)p}{4\epsilon}\int_\Omega u^p|\nabla v|^{2(q+1)}. \label{up-1-gradv2q-est}
\end{align}

Upon  integration by parts,  applications of   Young's inequality with $\epsilon$  and (\ref{C-S-ineq-c}), we find that
\begin{align}
&2\alpha q\int_\Omega u^p|\nabla v|^{2(q-1)}\nabla u\nabla v\nonumber\\
&=-\frac{2 (q-1)q\alpha}{p+1}\int_\Omega u^{p+1} |\nabla v|^{2(q-2)}\nabla v\cdot \nabla |\nabla v|^2-\frac{2 q\alpha}{p+1} \int_\Omega u^{p+1}|\nabla v|^{2(q-1)}\Delta v\nonumber\\
&\leq q(q-1)\eta \int_\Omega u^p|\nabla v|^{2(q-2)}| \nabla |\nabla v|^2|^2+\frac{ (q-1)q\alpha^2}{(p+1)^2\eta}\int_\Omega u^{p+2} |\nabla v|^{2(q-1)}\nonumber\\
&\quad +2d_2q\int_\Omega  u^p|\nabla v|^{2(q-1)}|D^2v|^2+\frac{qn\alpha^2}{2d_2(p+1)^2}\int_\Omega u^{p+2} |\nabla v|^{2(q-1)}. \label{lnter-u}\end{align}
Substituting \eqref{log-up-est},  \eqref{up-1-gradv2q-est} and \eqref{lnter-u} into \eqref{up-grdv2q-1}, after suitable rearrangements, we obtain the key inequality \eqref{up-grdv2q-f}.
\end{proof}

\section{Logistic damping prevents blow-up in 3-D setting}

To get a clear and better understanding about how strong a logistic damping can prevent blowup phenomenon in chemotaxis systems with  logistic sources, we first explore it in our physical world, i.e.,  in a three dimensional  bounded smooth domain $\Omega$. In this section, we will give details  to the algorithm leading to the  main result  (i) of Theorem \ref{main thm}.

 For the full-parameter  chemotaxis-growth system \eqref{para-para} in 3-D, we observe that,  to show the $L^\infty$-boundedness of $u$, it is enough to show the $L^{\frac{3}{2}+\epsilon}$-boundedness of $u$ for some $\epsilon>0$, thanks to the  boundedness criterion obtained in \cite{Xiangpre} via Moser iteration and in \cite{BBTW15} via semigroup theory.  This enables us to find out how large should $\mu$ be so that blow-up is impossible.  To achieve our goal, we need to  carefully collect the appearing constants in each derived inequalities.  We shall indeed prove that $\|u(t)\|_{L^2}$ is bounded. For this purpose, inspired  by \cite{Win10}, our analysis consists of deriving a delicate Gronwall inequality for the coupled functional
\be\label{couple-ineq}
z(t):=\delta_1\int_\Omega u^2(\cdot, t)+\delta_2\int_\Omega u(\cdot, t)|\nabla v(\cdot, t)|^2+\delta_3\int_\Omega |\nabla v(\cdot, t)|^4
\ee
of the form
$$
z^\prime(t)+\epsilon z(t)\leq C_\epsilon, \quad  t\in(0, T_{\text{max}})
$$
for some carefully chosen positive constants $\delta_1, \delta_2, \delta_3$,  $\epsilon>0$ and perhaps large $C_\epsilon$ independent of $t$.  Once this is done, the $L^2$-and hence the $L^\infty$-boundedness  of $u$ will be obtained.

Upon easy applications of Lemma \ref{u-lp-lem-optimal} with $p=2$ and Lemma \ref{gradv-2q-lem-optimal} with $q=2$, we obtain the following two lemmas.
\begin{lemma}  \label{u-l2-gradu-l2-lem} Let $f$ satisfy the logistic condition \eqref{log-con}. Then  the solution $(u,v)$ of \eqref{para-para} satisfies
\begin{equation}\label{u-l2-gradu-l2:est}
\frac{d}{dt}\int_\Omega u^2+2(d_1-\epsilon_1)\int_\Omega |\nabla u|^2 +2\mu\int_\Omega u^3\leq \frac{\chi^2}{2\epsilon_1}\int_\Omega u^2|\nabla v|^2+2a\int_\Omega u
\end{equation}
for all $t\in(0, T_{\text{max}})$ and for any $\epsilon_1\in(0, d_1)$.
\end{lemma}
\begin{lemma}  \label{gradv-l4-lem} The solution $(u,v)$ of \eqref{para-para} satisfies
\be\label{dgradvdt-int3-c}
\ba{ll}
&\frac{d}{dt}\int_\Omega|\bigtriangledown v|^4+2(d_2-\epsilon_2)\int_\Omega |\bigtriangledown |\bigtriangledown v|^2|^2+4\beta\int_\Omega |\bigtriangledown v|^4\\[0.25cm]
&\leq 2(\frac{n}{2d_2}+\frac{1}{\epsilon_2})\alpha^2\int_\Omega u^2|\bigtriangledown v|^2+2d_2\int_{\partial \Omega} |\bigtriangledown v|^2\frac{\partial}{\partial \nu} |\bigtriangledown v|^2
\ea
\ee
for all $t\in(0, T_{\text{max}})$ and for any $\epsilon_2\in(0, d_2)$.
\end{lemma}
\begin{lemma}\label{gradv-l4} Under the logistic condition \eqref{log-con}, the solution $(u,v)$ of of \eqref{para-para} satisfies
\be\label{ugrav-l2-lem}
\ba{ll}
&\frac{d}{dt}\int_\Omega u|\bigtriangledown v|^2+(\mu-\frac{\chi^2}{4\epsilon_3})\int_\Omega u^2|\nabla v|^2+2\beta \int_\Omega u |\nabla v|^2\\[0.25cm]
&\leq \frac{(d_1+d_2)^2}{4\epsilon_4}\int_\Omega |\nabla u|^2+(\epsilon_3+\epsilon_4)\int_\Omega  |\nabla|\nabla v|^2|^2+\frac{n\alpha^2}{8d_2}\int_\Omega u^3\\[0.25cm]
&\quad +a\int_\Omega |\nabla v|^2 +d_2\int_{\partial\Omega} u\frac{\partial}{\partial \nu} |\nabla v|^2
\ea
\ee
for all $t\in(0, T_{\text{max}})$ and for any $\epsilon_3, \epsilon_4>0$.
\end{lemma}
\begin{proof}Lemma \ref{key-lem-optimal} with $p=1=q$ reads as
\begin{align}
\frac{d}{dt} \int_\Omega u|\nabla v|^{2}&+2\beta \int_\Omega u |\nabla v|^{2}+\mu \int_\Omega u^{2} |\nabla v|^{2}\nonumber\\
&\leq\chi \int_\Omega u\nabla v\nabla |\nabla v|^2  -(d_1+d_2)\int_\Omega \nabla u\nabla |\nabla v|^2\nonumber\\
&\quad +\frac{n\alpha^2}{8d_2}\int_\Omega u^{3}+a\int_\Omega  |\nabla v|^{2}+d_2\int_{\partial \Omega}u\frac{\partial |\nabla v|^2}{\partial \nu}\label{u1-grdv21-f3}\end{align}
for all $t\in(0, T_{\text{max}})$.
Now, applying repeatedly Young's inequality with $\epsilon$ and taking into account \eqref{u-l2-gradu-l2:est}  and \eqref{dgradvdt-int3-c}, we infer that
\be\label{u-gradv-gradgradv-l2}
\chi \int_\Omega  u\nabla v\nabla|\nabla v|^2\leq\frac{ \chi^2}{4\epsilon_3} \int_\Omega  u^2|\nabla v|^2+\epsilon_3\int_\Omega|\nabla|\nabla v|^2|^2
\ee
as well as
\be\label{gradu-gradv-l2}
-(d_1+d_2)\int_\Omega \nabla u\nabla|\nabla v|^2\leq \frac{(d_1+d_2)^2}{4\epsilon_4}\int_\Omega |\nabla u|^2+\epsilon_4\int_\Omega|\nabla|\nabla v|^2|^2
\ee
for any $\epsilon_3,\epsilon_4>0$. Then we substitute  \eqref{u-gradv-gradgradv-l2} and \eqref{gradu-gradv-l2} into \eqref{u1-grdv21-f3} to conclude  \eqref{ugrav-l2-lem}.
\end{proof}

We are ready to study the coupled functional  \eqref{couple-ineq} by means of  Lemmas \ref{u-l2-gradu-l2-lem} -\ref{gradv-l4}.
\begin{lemma}\label{key-ineq-3d} Under the logistic condition \eqref{log-con},   the solution $(u,v)$ of \eqref{para-para} satisfies the inequality
\be\label{com-need-3d}
\begin{split}
&\frac{d}{dt}\Bigr\{\delta_1\int_\Omega u^2+\delta_2\int_\Omega u|\nabla v|^2+\delta_3\int_\Omega |\nabla v|^4\Bigr\}+(2\mu\delta_1-\frac{n\alpha^2}{8d_2}\delta_2)\int_\Omega u^3\\
&+[2(d_1-\epsilon_1)\delta_1-\frac{(d_1+d_2)^2}{4\epsilon_4} \delta_2]\int_\Omega |\nabla u|^2+4\beta\delta_3\int_\Omega |\nabla v|^4 \\
&\quad +[2(d_2-\epsilon_2)\delta_3-(\epsilon_3+\epsilon_4)\delta_2]\int_\Omega|\nabla |\nabla v|^2|^2+2\beta\delta_2\int_\Omega u|\nabla v|^2 \\
&\quad +[(\mu-\frac{\chi^2}{4\epsilon_3})\delta_2-2(\frac{n}{2d_2}+\frac{1}{\epsilon_2})\alpha^2\delta_3-\frac{\chi^2}{2\epsilon_1}\delta_1]\int_\Omega u^2|\nabla v|^2\\
&\leq 2a\delta_1\int_\Omega u+a\delta_2\int_\Omega |\nabla v|^2+2d_2\delta_3\int_{\partial \Omega} |\bigtriangledown v|^2\frac{\partial}{\partial \nu} |\bigtriangledown v|^2+d_2\delta_2\int_{\partial \Omega} u\frac{\partial}{\partial \nu} |\bigtriangledown v|^2
\end{split}
\ee
for any $t\in(0, T_{\text{max}})$ and any positive constants $\delta_1,\delta_2$ and $\delta_3$ and $\epsilon_1\in (0,d_1)$, $\epsilon_2\in (0, d_2) $ and $\epsilon_3, \epsilon_4>0$.
\end{lemma}
\begin{proof}By evident multiplications and additions from  Lemmas \ref{u-l2-gradu-l2-lem} -\ref{gradv-l4}, one can readily derive the inequality \eqref{com-need-3d}.
\end{proof}
Motivated by \eqref{com-need-3d}, to find the possibly smallest lower bound $\mu_0$ for the damping rate $\mu$ that could be obtained using such method, we  wish to choose $\epsilon_i$ and $\delta_i$ to satisfy
\begin{equation}\label{mu-min-deltas-3d-general}
\left\{ \begin{array}{ll}
&2(d_1-\epsilon_1)\delta_1-\frac{(d_1+d_2)^2}{4\epsilon_4} \delta_2>0, \\[0.25cm]
&2\mu\delta_1-\frac{n\alpha^2}{8d_2}\delta_2> 0, \\[0.25cm]
&2(d_2-\epsilon_2)\delta_3-(\epsilon_3+\epsilon_4)\delta_2> 0, \\[0.25cm]
&(\mu-\frac{\chi^2}{4\epsilon_3})\delta_2-2(\frac{n}{2d_2}+\frac{1}{\epsilon_2})\alpha^2\delta_3-\frac{\chi^2}{2\epsilon_1}\delta_1\geq 0. \end{array}\right.
\end{equation}
This can be easily satisfied and $\delta_2$ may be taken to be one.  The minimizer of the minimization problem \eqref{mu-min-deltas-3d-general} in $\mu$ will give us the smallest damping rate $\mu_0$ we are seeking.  Our goal is then  to choose  $\epsilon_i$ and $\delta_i$  so that $\mu$  is minimized.    By eliminations from \eqref{mu-min-deltas-3d-general}, we end up with
\be\label{mu0-min-3d}
\left\{ \begin{array}{ll}
&\mu>\frac{n\alpha^2}{16d_2}\frac{\delta_2}{\delta_1}, \\[0.25cm]
&\mu>\frac{\chi^2}{4\epsilon_3}+2(\frac{n}{2d_2}+\frac{1}{\epsilon_2})\alpha^2\cdot\frac{\epsilon_3+\epsilon_4}{2(d_2-\epsilon_2)}+\frac{\chi^2}{2\epsilon_1}\cdot\frac{\frac{(d_1+d_2)^2}{4\epsilon_4}}{2(d_1-\epsilon_1)} \end{array}\right.
\ee
for any $\epsilon_1\in (0,d_1), \epsilon_2\in (0,d_2)$ and $\epsilon_3,\epsilon_4>0$. Next, we shall first minimize the second expression on the  right-hand side of \eqref{mu0-min-3d}.  Notice that
\be\label{mu0-min-3d-2}
\frac{\chi^2}{2\epsilon_1}\cdot\frac{\frac{(d_1+d_2)^2}{4\epsilon_4}}{2(d_1-\epsilon_1)}\geq \frac{\chi^2 }{d_1^2}\frac{(d_1+d_2)^2}{4\epsilon_4}
\ee
with equality if and only $\epsilon_1=d_1/2$, and
\be\label{mu0-min-3d-3}
2(\frac{n}{2d_2}+\frac{1}{\epsilon_2})\alpha^2\cdot\frac{\epsilon_3+\epsilon_4}{2(d_2-\epsilon_2)}\geq \Bigr[ \frac{n\alpha}{(\sqrt{2n+4}- 2)d_2}\Bigr]^2(\epsilon_3+\epsilon_4)
\ee
with equality if and only if
$$
 \epsilon_2=\frac{(\sqrt{2n+4}-2)}{n}d_2.
$$
Therefore,  by algebraic calculations  from \eqref{mu0-min-3d}-\eqref{mu0-min-3d-3}  we deduce that the second expression on the right-hand side of \eqref{mu0-min-3d} achieves its minimum
\be\label{mu0-3d-1}
\frac{n}{\sqrt{2n+4}-2}(\frac{1}{d_1}+\frac{2}{d_2})\alpha |\chi|
\ee
if and only if
\begin{equation}\label{epsilon-is}
\left\{ \begin{array}{ll}
&\epsilon_1=\frac{d_1}{2}, \quad \quad \quad \quad \quad  \quad \quad \epsilon_2=\frac{(\sqrt{2n+4}-2)}{n}d_2,  \\[0.25cm]
&\epsilon_3=\frac{(\sqrt{2n+4}-2)d_2}{2n\alpha}|\chi|,\quad\quad \epsilon_4= \frac{(\sqrt{2n+4}-2)(d_1+d_2)d_2}{2nd_1\alpha}|\chi|. \end{array}\right.
\end{equation}
For such well-chosen  $\epsilon_i$ according to \eqref{epsilon-is}, we  choose $\delta_i$ to fulfill the first two inequalities in \eqref{mu-min-deltas-3d-general}  as well as
$$
\frac{n\alpha^2}{16d_2}\frac{\delta_2}{\delta_1}<\frac{n}{\sqrt{2n+4}-2}(\frac{1}{d_1}+\frac{2}{d_2})\alpha |\chi|,
$$
then we infer  from \eqref{mu0-min-3d}-\eqref{mu0-3d-1} that
\be\label{mu0-min-final1}
\mu>\mu_0(n,d_1,d_2,\alpha,\chi):=\frac{n}{\sqrt{2n+4}-2}(\frac{1}{d_1}+\frac{2}{d_2})\alpha |\chi|.
\ee
Indeed,  for any $\mu$ satisfying \eqref{mu0-min-final1}, we fix $\epsilon_i$ according to \eqref{epsilon-is} and $\delta_i$ according to
\begin{equation}\label{delta-is}
\left\{ \begin{array}{ll}
&\delta_1=\max\Bigr\{\frac{(d_1+d_2)^2}{8\epsilon_4(d_1-\epsilon_1)}, \frac{\alpha^2}{16d_2}\frac{(\sqrt{2n+4}-2)}{(\frac{1}{d_1}+\frac{2}{d_2})\alpha |\chi|}\Bigr\}+1,  \\[0.25cm]
& \delta_2=1,\quad  \quad \quad  \quad  \delta_3=\frac{\epsilon_3+\epsilon_4}{2(d_2-\epsilon_2)}+1,  \end{array}\right.
\end{equation}
then the inequality \eqref{mu-min-deltas-3d-general} will be satisfied.

For any   $\mu>\mu_0$ as defined in \eqref{mu0-min-final1}, we shall illustrate that all integrals on the right-hand side of \eqref{com-need-3d} can be controlled by the dissipative terms on the left and, as a result,  yielding the assertion that $\|u(t)\|_{L^2}$ is uniformly bounded and thus  (i) of Theorem \ref{main thm} in physically relevant setting $n=3$ by the $L^{\frac{n}{2}+\epsilon}$-criterion in \cite{BBTW15,Xiangpre}.
\begin{lemma} \label{main-thm-3d}  Let $\Omega\subset \mathbb{R}^ n$ be a bounded smooth domain,  $f$ satisfy  the logistic condition \eqref{log-con} and $d_1,d_2,\alpha,\beta>0$, $a\geq 0$ and $\chi\in\mathbb{R}$. Then, for any  $\mu>\mu_0(n,d_1,d_2,\alpha,\chi)$ as given in \eqref{mu0-min-final1},  there exists a constant $C(u_0, v_0)$ such that the unique nonnegative solution $(u,v)$ of the chemotaxis-growth system \eqref{para-para} satisfies
$$
\int_\Omega u^2(\cdot, t)+\int_\Omega u(\cdot, t)|\nabla v(\cdot, t)|^2+\int_\Omega |\nabla v(\cdot, t)|^4\leq C(u_0, v_0), \quad \quad t\in (0, T_{\text{max}}).
$$
Thus, if $n\leq 3$, then the solution $(u,v)$ exists globally in time, i.e.,  $T_{\text{max}}=\infty$ and  $(u(\cdot,t), v(\cdot,t))$ is uniformly bounded in  $L^\infty(\Omega)\times W^{1,\infty}(\Omega)$ for all $t\in (0, \infty)$.
\end{lemma}
\begin{proof}
 It follows from  Lemma \ref{key-ineq-3d}  that the quantity
$$
z(t):=\delta_1\int_\Omega u^2+\delta_2\int_\Omega u|\nabla v|^2+\delta_3\int_\Omega |\nabla v|^4, \quad \quad t\in (0, T_{\text{max}}),
$$
fulfills
\be\label{differ-ineq-3d}
\begin{array}{ll}
&z^\prime (t)+\beta z(t)+[2(d_1-\epsilon_1)\delta_1-\frac{(d_1+d_2)^2}{4\epsilon_4}\delta_2]\int_\Omega |\nabla u|^2\\[0.25cm]
&\quad \quad +[2(d_2-\epsilon_2)\delta_3-(\epsilon_3+\epsilon_4)\delta_2]\int_\Omega|\nabla |\nabla v|^2|^2+3\beta\delta_3\int_\Omega |\nabla v|^4\\[0.25cm]
&\quad \quad +\beta\delta_2\int_\Omega u|\nabla v|^2+(2\mu\delta_1-\frac{n\alpha^2}{8d_2}\delta_2)\int_\Omega u^3\\[0.25cm]
&\quad\quad  \leq 2a\delta_1\int_\Omega u+a\delta_2\int_\Omega |\nabla v|^2+\beta\delta_1\int_\Omega u^2\\[0.25cm]
&\quad \quad \quad+2d_2\delta_3\int_{\partial \Omega} |\bigtriangledown v|^2\frac{\partial}{\partial \nu} |\bigtriangledown v|^2+d_2\delta_2\int_{\partial \Omega} u\frac{\partial}{\partial \nu} |\bigtriangledown v|^2 .
\end{array}
\ee
 In the sequel,  we bound  the integrals on the right-hand side of \eqref{differ-ineq-3d} in terms of  the dissipative terms on its left-hand side.

Let us first focus on controlling the boundary integrals in \eqref{differ-ineq-3d}. So far, there are a couple of existing ways to handle these boundary integrals,  see \cite{ISY14, LX14,  TWZAMP, ZL15-ZAMP}, for instance. Here we would like to provide an alternative and transparent way  to remove the technical assumption that the  domain $\Omega$ be convex.  The starting point is based on the  pointwise geometric inequality
\be\label{geometric-ineq}
\frac{\partial |\nabla w|^2}{\partial \nu}\leq K_1(\Omega)|\nabla w|^2 \quad \text{on } \partial \Omega,
\ee
which holds  for any bounded smooth  domain $\Omega\subset \mathbb{R}^n$ and  any $w$ satisfying $\frac{\partial w}{\partial \nu}=0$ on $\partial \Omega$, cf. \cite[Lemma 4.2]{MS14}. Here and below, $K_i$ will denote some inessential constants.  Notice also a user-friendly  version of trace inequality with $\epsilon$ (cf. \cite[Remark 52.9]{QSbook}): for any $\epsilon>0$, one has
\be\label{trace-ineq-epsilon0}
\|w\|_{L^2(\partial\Omega)}\leq \epsilon \|\nabla w\|_{L^2(\Omega)}+C_\epsilon \|w\|_{L^2(\Omega)}, \quad \forall w\in H^1(\Omega).
\ee
Indeed, this is immediately implied, upon a use of Young's inequality with $\epsilon$,  by the following version of trace inequality:
\be\label{trace-ineq-epsilon2}
\|w\|_{L^2(\partial\Omega)}\leq  K_2 \| w\|_{L^2(\Omega)}^{\frac{1}{2}} \|w\|_{H^1(\Omega)}^{\frac{1}{2}}, \quad \forall w\in H^1(\Omega).
\ee
As a matter of fact, by the property of trace inequality (the trace operator $T$ maps $H^{\frac{1}{2}}(\Omega)$ continuously  onto $L^2(\partial\Omega)$),  one has
$$
\|w\|_{L^2(\partial \Omega)}\leq K_3\|w\|_{H^{\frac{1}{2}}(\Omega)}, \quad \forall w\in H^{\frac{1}{2}}(\Omega).
$$
On the other hand, it follows from the fact that $H^{\frac{1}{2}}$ interpolates the spaces $H^0=L^2$ and $H^1$ that
$$
\|w\|_{H^{\frac{1}{2}}(\Omega)}\leq K_4\|w\|_{L^2(\Omega)}^{\frac{1}{2}}\|w\|_{H^1(\Omega)}^{\frac{1}{2}}, \quad \forall w\in H^1(\Omega).
$$
Collecting these two estimates directly  leads to \eqref{trace-ineq-epsilon2}.

Since $H^1(\Omega)$ is compactly embedded in $L^2(\Omega)$ by Kondrachov and that $L^2(\Omega)$ is continuously embedded in $L^1(\Omega)$, Lion's lemma says,  for any $\eta>0$, that
$$
\|w\|_{L^2(\Omega)}\leq \eta \|\nabla w\|_{L^2(\Omega)}+\eta\|w\|_{L^2(\Omega)}+C_\eta\|w\|_{L^1(\Omega)},  \quad \forall w\in H^1(\Omega).
$$
Combing this with \eqref{trace-ineq-epsilon0}, we conclude, for any $\epsilon>0$,
\be\label{trace-ineq-epsilon}
\|w\|_{L^2(\partial\Omega)}\leq \epsilon \|\nabla w\|_{L^2(\Omega)}+C_\epsilon \|w\|_{L^1(\Omega)}, \quad \forall w\in H^1(\Omega).
\ee
This gives another elementary  proof for the equivalent  trace inequality stated in \cite[P. 13, line -4]{TWZAMP}.

Now, with \eqref{geometric-ineq} and  \eqref{trace-ineq-epsilon} at hand, we can cope with  the boundary integrals in \eqref{differ-ineq-3d} as follows: for any $\epsilon>0$,
\be\label{boundary-integral-3d}
\begin{array}{ll}
&2d_2\delta_2\int_{\partial \Omega} |\bigtriangledown v|^2\frac{\partial}{\partial \nu} |\bigtriangledown v|^2+d_2\delta_3\int_{\partial \Omega} u\frac{\partial}{\partial \nu} |\bigtriangledown v|^2\\[0.25cm]
& \leq 2d_2\delta_2 K_1\int_{\partial \Omega} |\bigtriangledown v|^4+d_2\delta_3K_1\int_{\partial \Omega} u |\bigtriangledown v|^2\\[0.25cm]
&\leq K_5\int_{\partial \Omega} |\bigtriangledown v|^4+K_6\int_{\partial \Omega} u^2=K_5\||\nabla v|^2\|_{L^2(\partial \Omega)}^2+K_6\|u\|_{L^2(\partial \Omega)}^2\\[0.25cm]
&\leq \epsilon\int_\Omega |\nabla|\nabla v|^2|^2+C_ \epsilon  \int_{\Omega} |\bigtriangledown v|^2+ \epsilon  \int_\Omega |\nabla u|^2+C_ \epsilon \int_{\Omega} u.
\end{array}
\ee
For any  $\mu>\mu_0(n,d_1,d_2,\alpha,\chi)$ as given in \eqref{mu0-min-final1}, we first fix $\epsilon_i, i=1,2,3,4$ according to \eqref{epsilon-is}, and then  fix  $\delta_1,\delta_2$  and $\delta_3$ complying with \eqref{delta-is}. Then the inequality \eqref{mu-min-deltas-3d-general} is satisfied.  Upon such selections,  we then fix $\epsilon$  according to
$$
\epsilon=\frac{1}{2}\min\Bigr\{2(d_1-\epsilon_1)\delta_1-\frac{(d_1+d_2)^2}{4\epsilon_4}\delta_3,\quad 2(d_2-\epsilon_2)\delta_3-(\epsilon_3+\epsilon_4)\delta_2\Bigr\}.
$$
Then the boundary integrals in \eqref{boundary-integral-3d}  will be absorbed by the terms on the  left-hand side of \eqref{differ-ineq-3d} and Lemma \ref{ul1-vgradl2}.

Next, notice that
\be\label{u3-dominate-u2}
 \beta\delta_1\int_\Omega u^2\leq (2\mu\delta_1-\frac{n\alpha^2}{8d_2}\delta_2)\int_\Omega u^3+K_7|\Omega|,
\ee
where
$$
K_7=\max\{\beta\delta_1u^2-(2\mu\delta_1-\frac{n\alpha^2}{8d_2}\delta_2) u^3|u\geq 0\}<\infty.
$$

Finally,  we substitute \eqref{boundary-integral-3d}  and \eqref{u3-dominate-u2} into \eqref{differ-ineq-3d} and use Lemma \ref{ul1-vgradl2} to conclude
 $$
z^\prime (t)+\beta z(t)\leq C(u_0,v_0),  \quad \quad t\in [0, T_{\text{max}}),
$$
which together the definition of $z$ simply leads to
$$
\ba{ll}
z(t)&=\delta_1\int_\Omega u^2+\delta_2\int_\Omega u|\nabla v|^2+\delta_3\int_\Omega |\nabla v|^4\\[0.25cm]
&\leq \delta_1\int_\Omega u_0^2+\delta_2\int_\Omega u_0|\nabla v_0|^2+\delta_3\int_\Omega |\nabla v_0|^4
+\frac{C(u_0,v_0)}{\beta},  \quad t\in [0, T_{\text{max}}).
\ea
$$
This in particular shows the uniform spatial  $L^2$-boundedness of $u(\cdot, t)$ for any $t\in [0, T_{\mbox{max}})$. Hence, by Moser iteration, cf.  the   $L^{\frac{n}{2}+\epsilon}$-criterion in \cite[R1, Theorem 1.1]{Xiangpre}  or \cite[Lemma 3.2]{BBTW15} with $n=3$, we infer that $T_{\text{max}}=\infty$ and that $(u(\cdot,t), v(\cdot,t))$ is uniformly bounded in $L^\infty(\Omega)\times W^{1,\infty}(\Omega)$ for all $t\in (0, \infty)$.
\end{proof}

\section{How strong a logistic damping can prevent blowups in $n$-D}
\subsection{Blow-up prevention by logistic source in nonconvex domains}

\ \ In this subsection,  based on the detailed algorithm in $3$-D, we wish to provide a clue on how to compute the explicit logistic damping rate $\mu_0$ that suppresses blow-up whenever $\mu>\mu_0$ in $\geq4$-D. We will mainly do it for $n=4, 5$. Our procedure suggests that an  explicit logistic damping rate $\mu_0$ suppressing blow-up whenever $\mu>\mu_0$, which enjoys the properties as described in Remark \ref{note on main thm},  is also available in more higher dimensions.

When $n\leq 5$, it amounts to ensuring  the uniform boundedness of $\|u\|_{L^3(\Omega)}$. In such setup,  the core analysis then lies in deriving a subtle estimate for
\be\label{couple-ineq-4d}
z(t):=\delta_1\int_\Omega u^3+\delta_2\int_\Omega u^2|\nabla v|^2+\delta_3\int_\Omega u|\nabla v|^4+\delta_4\int_\Omega |\nabla v|^6, \quad t\in(0, T_{\text{max}})
\ee
of the form
\be\label{gronwall-ineq-4d}
z^\prime(t)+\zeta z(t)\leq C_\zeta , \quad  t\in(0, T_{\text{max}})
\ee
for some well chosen positive constants $\delta_i$ and $\zeta >0$ and perhaps large $C_\zeta $ independent of $t$.  Once this is done, the $L^3$-and hence the $L^\infty$-boundedness  of $u$ will be obtained, by the boundedness principles \cite{BBTW15, Xiangpre}.

Upon trivial applications of Lemma \ref{u-lp-lem-optimal} with $p=3$ and Lemma \ref{gradv-2q-lem-optimal} with $q=3$, we get the following two lemmas.
\begin{lemma}  \label{u-l3-gradu-l3-lem-45d} Let $f$ satisfy the logistic condition \eqref{log-con}. Then  the solution $(u,v)$ of \eqref{para-para} satisfies
\begin{equation}\label{u-l3-gradu-l3:est}
\frac{d}{dt}\int_\Omega u^3+6(d_1-\epsilon)\int_\Omega u|\nabla u|^2 +3\mu\int_\Omega u^4\leq \frac{3\chi^2}{2\epsilon}\int_\Omega u^3|\nabla v|^2+3a\int_\Omega u^2
\end{equation}
for all $t\in(0, T_{\text{max}})$ and for any $\epsilon\in(0, d_1)$.
\end{lemma}
\begin{lemma}  \label{gradv-l6-lem-45d} The solution $(u,v)$ of \eqref{para-para} satisfies
\begin{align}\label{v-grdv-23}
\frac{d}{dt} \int_\Omega |\nabla v|^{6}+&6(d_2-\eta)\int_\Omega  |\nabla v|^{2}|\nabla|\nabla v|^2|^2+6\beta\int_\Omega |\nabla v|^{6}\nonumber\\
&\leq 3(\frac{ 2\alpha^2}{\eta}+\frac{n\alpha^2}{2d_2})\int_\Omega u^2 |\nabla v|^{4}+3d_2\int_{\partial\Omega}  |\nabla v|^{4}\frac{\partial}{\partial \nu} |\nabla v|^2 \end{align}
for all $t\in(0, T_{\text{max}})$ and for any $\epsilon_2\in(0, d_2)$.
\end{lemma}

\begin{lemma}  \label{key-lem-optimal p=1} Let $f$ satisfy the logistic condition \eqref{log-con}. Then the solution $(u,v)$ of \eqref{para-para} verifies
\begin{align}
\frac{d}{dt} \int_\Omega u|\nabla v|^{4}&+4\beta \int_\Omega u |\nabla v|^{4}\nonumber\\
&+(\mu-\frac{\chi^2}{2\epsilon_1}) \int_\Omega u^2 |\nabla v|^{4}+2(d_2-\eta)\int_\Omega u|\nabla |\nabla v|^2|^2\nonumber\\
&\leq 2(\epsilon_1+\epsilon_2)\int_\Omega |\nabla v|^{2}|\nabla |\nabla v|^2|^2+\frac{(d_1+d_2)^2}{2\epsilon_2}\int_\Omega |\nabla u|^2|\nabla v|^2\nonumber\\
&\quad  +\frac{\alpha^2}{2}(\frac{1}{\eta}+\frac{n}{2d_2})\int_\Omega u^3|\nabla v|^{2}\nonumber\\
 &\quad +a\int_\Omega |\nabla v|^{4}+2d_2\int_{\partial \Omega}u|\nabla v|^{2}\frac{\partial |\nabla v|^2}{\partial \nu}\label{u1-grdv22-f}\end{align}
for all $t\in(0, T_{\text{max}})$ and for any $\epsilon\in (0, d_1)$,  $\eta\in(0, d_2)$ and $\epsilon_1, \epsilon_2>0$.
\end{lemma}
\begin{proof} Lemma \ref{key-lem-optimal} with  $p=1$ and $q=2$ reads as
\begin{align}
\frac{d}{dt} \int_\Omega u|\nabla v|^{4}&+4\beta \int_\Omega u|\nabla v|^{4}\nonumber\\
&+2(d_2-\eta)\int_\Omega u|\nabla |\nabla v|^2|^2+\mu \int_\Omega u^{2} |\nabla v|^{4}\nonumber\\
&\leq2\chi \int_\Omega u|\nabla v|^{2}\nabla v\nabla |\nabla v|^2 \nonumber\\
& \quad -2(d_1+d_2)\int_\Omega |\nabla v|^{2}\nabla u\nabla |\nabla v|^2\nonumber\\
&\quad +\frac{\alpha^2}{2}(\frac{1}{\eta}+\frac{n}{2d_2})\int_\Omega u^{3}|\nabla v|^{2}\nonumber\\
 &\quad +a\int_\Omega |\nabla v|^{4}+2d_2\int_{\partial \Omega}u|\nabla v|^{2}\frac{\partial |\nabla v|^2}{\partial \nu}.\label{u1-grdv22}\end{align}
Taking into account \eqref{v-grdv-23} of Lemma \ref{gradv-l6-lem-45d},  we estimate
\be\label{p=1 q=2:1}
2\chi \int_\Omega u|\nabla v|^{2}\nabla v\nabla |\nabla v|^2\leq 2\epsilon_1\int_\Omega |\nabla v|^{2}|\nabla |\nabla v|^2|^2+\frac{\chi^2}{2\epsilon_1}\int_\Omega u^2|\nabla v|^{4}
\ee
and
\be\label{p=1 q=2:2}
\begin{split}
&-2(d_1+d_2)\int_\Omega |\nabla v|^{2}\nabla u\nabla |\nabla v|^2\\
&\ \  \leq 2\epsilon_2\int_\Omega |\nabla v|^{2}|\nabla |\nabla v|^2|^2+\frac{(d_1+d_2)^2}{2\epsilon_2}\int_\Omega |\nabla u|^2|\nabla v|^2.
\end{split}
\ee
Then combining  \eqref{u1-grdv22}, \eqref{p=1 q=2:1} and \eqref{p=1 q=2:2}, we get the desired inequality \eqref{u1-grdv22-f}.
\end{proof}
\begin{lemma}  \label{key-lem-optimal p=2 q=1} Let $f$ satisfy the logistic condition \eqref{log-con}. Then  the solution $(u,v)$ of \eqref{para-para} verifies
\begin{align}
\frac{d}{dt} \int_\Omega u^2|\nabla v|^{2}&+2(d_1-\epsilon)\int_\Omega |\nabla u|^2|\nabla v|^{2}+2\beta \int_\Omega u^2 |\nabla v|^{2}\nonumber\\
&+(2\mu-\frac{\chi^2}{2\epsilon_3}) \int_\Omega u^{3} |\nabla v|^{2}\nonumber\\
&\leq\frac{\chi^2 }{2\epsilon}\int_\Omega u^2|\nabla v|^{4}+2\epsilon_4\int_\Omega u |\nabla u|^2 \nonumber\\
& \quad +\Bigr[2\epsilon_3+\frac{(d_1+d_2)^2}{2\epsilon_4}\Bigr]\int_\Omega u|\nabla |\nabla v|^2|^2\nonumber\\
 &\quad +\frac{n\alpha^2}{18d_2}\int_\Omega u^4+2a\int_\Omega u |\nabla v|^{2}+d_2\int_{\partial \Omega}u^2\frac{\partial |\nabla v|^2}{\partial \nu}\label{u2-grdv21-f}\end{align}
for all $t\in(0, T_{\text{max}})$ and for any $\epsilon\in (0, d_1)$,  $\eta\in(0, d_2)$ and $\epsilon_3, \epsilon_4>0$.
\end{lemma}
\begin{proof}It follows from  \eqref{up-grdv2q-f} with $p=2$ and $q=1$ in Lemma \ref{key-lem-optimal} that
\begin{align}
\frac{d}{dt} \int_\Omega u^2|\nabla v|^{2}&+2(d_1-\epsilon)\int_\Omega |\nabla u|^2|\nabla v|^{2}+2\beta \int_\Omega u^2|\nabla v|^{2}\nonumber\\
&+2\mu \int_\Omega u^{3} |\nabla v|^{2}\nonumber\\
&\leq\frac{\chi^2 }{2\epsilon}\int_\Omega u^2|\nabla v|^{4}+2\chi \int_\Omega u^2\nabla v\nabla |\nabla v|^2 \nonumber\\
& \quad -2(d_1+d_2)\int_\Omega u\nabla u\nabla |\nabla v|^2+\frac{\alpha^2}{9}\frac{n}{2d_2}\int_\Omega u^{4}\nonumber\\
 &\quad +2a\int_\Omega u |\nabla v|^{2}+d_2\int_{\partial \Omega}u^2\frac{\partial |\nabla v|^2}{\partial \nu}.\label{u2-grdv21}\end{align}
Taking into consideration \eqref{u1-grdv22-f} and \eqref{u-l3-gradu-l3:est}, we bound
\be\label{p=2 q=1:1}
2\chi \int_\Omega u^{2}\nabla v\nabla |\nabla v|^2\leq 2\epsilon_3\int_\Omega u|\nabla |\nabla v|^2|^2+\frac{\chi^2}{2\epsilon_3} \int_\Omega u^{3} |\nabla v|^{2}
\ee
and
\be\label{p=2 q=1:2}
-2(d_1+d_2)\int_\Omega u\nabla u\nabla |\nabla v|^2\leq 2\epsilon_4\int_\Omega u|\nabla u|^2+\frac{(d_1+d_2)^2}{2\epsilon_4}\int_\Omega u|\nabla |\nabla v|^2|^2.
\ee
Then using \eqref{u2-grdv21}, \eqref{p=2 q=1:1} and \eqref{p=2 q=1:2}, we deduce \eqref{u2-grdv21-f} and hence  the lemma.
\end{proof}

At this moment,  we are well-prepared to estimate the coupled functional  \eqref{couple-ineq-4d} by means of  Lemmas \ref{u-l3-gradu-l3-lem-45d} -\ref{key-lem-optimal p=2 q=1}.
\begin{lemma}\label{key-ineq--45d} Under the logistic condition \eqref{log-con}, the solution $(u,v)$ of \eqref{para-para} satisfies the inequality
\be\label{com-need-4d}
\begin{split}
&\frac{d}{dt}\Bigr\{\delta_1\int_\Omega u^3+\delta_2\int_\Omega u^2|\nabla v|^2+\delta_3\int_\Omega u|\nabla v|^4+\delta_4\int_\Omega |\nabla v|^6\Bigr\}+6\beta\delta_4\int_\Omega |\nabla v|^6\\
&\quad +2[3(d_2-\eta)\delta_4-(\epsilon_1+\epsilon_2)\delta_3]\int_\Omega |\nabla v|^{2}|\nabla |\nabla v|^2|^2+(3\mu\delta_1-\frac{n\alpha^2}{18d_2}\delta_2)\int_\Omega u^4\\
&\quad +[2(d_1-\epsilon)\delta_2-\frac{(d_1+d_2)^2}{2\epsilon_2}\delta_3]\int_\Omega |\nabla u|^2|\nabla v|^2+2\beta \delta_2\int_\Omega u^2|\nabla v|^2\\
&\quad+ \Bigr[(2(d_2-\eta)\delta_3-(2\epsilon_3+\frac{(d_1+d_2)^2}{2\epsilon_4})\delta_2\Bigr]\int_\Omega u|\nabla |\nabla v|^2|^2\\
&\quad +4\beta\delta_3\int_\Omega u|\nabla v|^4+2[3(d_1-\epsilon)\delta_1-\epsilon_4\delta_2]\int_\Omega u |\nabla u|^2\\
&\quad+ \Bigr[(2\mu-\frac{\chi^2}{2\epsilon_3})\delta_2- \frac{3\chi^2}{2\epsilon}\delta_1-\frac{\alpha^2}{2}(\frac{1}{\eta}+\frac{n}{2d_2})\delta_3\Bigr]\int_\Omega u^{3} |\nabla v|^{2}\\
&\quad+ \Bigr[(\mu-\frac{\chi^2}{2\epsilon_1})\delta_3- \frac{\chi^2}{2\epsilon}\delta_2-3(\frac{2\alpha^2}{\eta}+\frac{n\alpha^2}{2d_2})\delta_4\Bigr]\int_\Omega u^{2} |\nabla v|^{4}\\
&\leq 3a\delta_1\int_\Omega u^2+2a\delta_2\int_\Omega u |\nabla v|^2+a\delta_3\int_\Omega |\nabla v|^4+3d_2\delta_4\int_{\partial\Omega}  |\nabla v|^{4}\frac{\partial}{\partial \nu} |\nabla v|^2\\
&\quad +2d_2\delta_3\int_{\partial \Omega} u|\bigtriangledown v|^2\frac{\partial}{\partial \nu} |\bigtriangledown v|^2+d_2\delta_2\int_{\partial \Omega} u^2\frac{\partial}{\partial \nu} |\bigtriangledown v|^2
\end{split}
\ee
for all $t\in(0, T_{\text{max}})$ and any positive constants $\epsilon_i, \delta_i$ and $\epsilon\in (0,d_1)$ and $\eta\in (0, d_2) $.
\end{lemma}
\begin{proof}By honest computations from Lemmas \ref{u-l3-gradu-l3-lem-45d} -\ref{key-lem-optimal p=2 q=1} and by evident multiplications and additions, one can readily derive the lemma.
\end{proof}
As to the boundary integrals on the right-hand side of \eqref{com-need-4d},  we deduce from \eqref{geometric-ineq},  \eqref{trace-ineq-epsilon}  and Young's inequality with epsilon that
\be\label{bddary int n=4}\begin{split}
&3d_2\delta_4\int_{\partial\Omega}  |\nabla v|^{4}\frac{\partial}{\partial \nu} |\nabla v|^2+2d_2\delta_3\int_{\partial \Omega} u|\bigtriangledown v|^2\frac{\partial}{\partial \nu} |\bigtriangledown v|^2+d_2\delta_2\int_{\partial \Omega} u^2\frac{\partial}{\partial \nu} |\bigtriangledown v|^2\\
&\leq  3d_2\delta_4K_1\int_{\partial\Omega}  |\nabla v|^{6} +2d_2\delta_3 K_1\int_{\partial \Omega} u|\bigtriangledown v|^4+d_2\delta_2K_1\int_{\partial \Omega} u^2 |\bigtriangledown v|^2\\
&\leq K_{8}\int_{\partial \Omega} u^3+K_{8} \int_{\partial \Omega} |\bigtriangledown v|^6=K_8\|u^{\frac{3}{2}}\|_{L^2(\partial \Omega)}^2++K_8 \| |\bigtriangledown v|^3\|_{L^2(\partial \Omega)}^2\\
&\leq \xi \int_\Omega u|\nabla u|^2+C_\xi\Bigr(\int_\Omega u^{\frac{3}{2}}\Bigr)^2+ \sigma \int_\Omega  |\nabla v|^{2}|\nabla |\nabla v|^2|^2+C_\sigma \Bigr(\int_\Omega  |\nabla v|^3\Bigr)^2
\end{split}
\ee
for any $\xi>0$ and  $\sigma>0$.

Based on this boundary integral estimate and  \eqref{com-need-4d}, we wish to select $\epsilon, \eta, \epsilon_i, \delta_i$  and $\mu$ such that
\begin{equation}\label{mu-min-deltas-4d-general}
\left\{ \begin{array}{ll}
&2[3(d_1-\epsilon)\delta_1-\epsilon_4\delta_2]>0, \\[0.25cm]
&2[3(d_2-\eta)\delta_4-(\epsilon_1+\epsilon_2)\delta_3]> 0, \\[0.25cm]
&2(d_1-\epsilon)\delta_2-\frac{(d_1+d_2)^2}{2\epsilon_2}\delta_3\geq 0, \\[0.25cm]
 &2(d_2-\eta)\delta_3-(2\epsilon_3+\frac{(d_1+d_2)^2}{2\epsilon_4})\delta_2\geq 0, \\[0.25cm]
&3\mu\delta_1-\frac{n\alpha^2}{18d_2}\delta_2\geq 0,\\[0.25cm]
&(2\mu-\frac{\chi^2}{2\epsilon_3})\delta_2- \frac{3\chi^2}{2\epsilon}\delta_1-\frac{\alpha^2}{2}(\frac{1}{\eta}+\frac{n}{2d_2})\delta_3\geq 0, \\[0.25cm]
&(\mu-\frac{\chi^2}{2\epsilon_1})\delta_3- \frac{\chi^2}{2\epsilon}\delta_2-3(\frac{2\alpha^2}{\eta}+\frac{n\alpha^2}{2d_2})\delta_4\geq 0.\end{array}\right.
\end{equation}
Only the third constraint  intertwines with the fourth constraint, while this  can be fulfilled  since we can choose $\epsilon, \eta, \epsilon_i$ such that
\be\label{delta2-delta3-re}
\frac{(d_1+d_2)^2}{4\epsilon_2(d_1-\epsilon)}\leq \frac{2(d_2-\eta)}{2\epsilon_3+\frac{(d_1+d_2)^2}{2\epsilon_4}}\Longleftrightarrow \frac{(d_1+d_2)^2}{4(d_1-\epsilon)(d_2-\eta)}\leq \frac{\epsilon_2}{\epsilon_3+\frac{(d_1+d_2)^2}{4\epsilon_4}}< \frac{\epsilon_2}{\epsilon_3}.
\ee
Then algebraic manipulations from  \eqref{mu-min-deltas-4d-general} and \eqref{delta2-delta3-re} show that
\begin{equation}\label{mu>mu0-4d0}
\ba{ll}
3\mu&\geq \frac{n\alpha^2}{54d_2}\frac{\delta_2}{\delta_1}+ \frac{\chi^2}{4\epsilon_3}+ \frac{3\chi^2}{4\epsilon}\frac{\delta_1}{\delta_2}+\frac{\alpha^2}{4}(\frac{1}{\eta}+\frac{n}{2d_2})\frac{\delta_3}{\delta_2}+\frac{\chi^2}{2\epsilon_1}+ \frac{\chi^2}{2\epsilon}\frac{\delta_2}{\delta_3}+3(\frac{2\alpha^2}{\eta}+\frac{n\alpha^2}{2d_2})\frac{\delta_4}{\delta_3}\\[0.25cm]
&> \sqrt{\frac{n}{18d_2\epsilon}}\alpha |\chi|+ \frac{\chi^2}{4\epsilon_3}+\sqrt{\frac{1}{2\epsilon}(\frac{1}{\eta}+\frac{n}{2d_2})}\alpha |\chi|+\frac{\chi^2}{2\epsilon_1}+(\frac{2\alpha^2}{\eta}+\frac{n\alpha^2}{2d_2})\frac{\epsilon_1+\epsilon_2}{d_2-\eta}\\[0.25cm]
&> \Big\{\sqrt{\frac{n}{18d_2\epsilon}}+\sqrt{\frac{1}{2\epsilon}(\frac{1}{\eta}+\frac{n}{2d_2})}+\sqrt{\frac{1}{(d_2-\eta)}(\frac{2}{\eta}+\frac{n}{2d_2})}\Bigr[\sqrt{2}+\frac{(d_1+d_2)}{2\sqrt{(d_1-\epsilon)(d_2-\eta)}}\Bigr]\Bigr\}\alpha |\chi|
\ea
\end{equation}
\begin{lemma}\label{main-thm-4d}  Let $\Omega\subset \mathbb{R}^ n$ be a bounded smooth domain,  $f$ satisfy  the logistic condition \eqref{log-con} and $d_1,d_2,\alpha,\beta>0$, $a\geq 0$ and $\chi\in\mathbb{R}$. Then, for any
\be\label{mu>mu0-4d}
\mu>\mu_0(n,d_1,d_2,\alpha,\chi):=\max\Bigr\{\frac{1}{3}h(n,d_1,d_2), \frac{n}{\sqrt{2n+4}-2}(\frac{1}{d_1}+\frac{2}{d_2})\Bigr\}\alpha|\chi|
\ee
with
$$
\ba{ll}
h(n,d_1,d_2)=\min_{0<\epsilon<d_1, 0<\eta<d_2}&\Big\{\sqrt{\frac{n}{18d_2\epsilon}}+\sqrt{\frac{1}{2\epsilon}(\frac{1}{\eta}+\frac{n}{2d_2})}\\[0.25cm]
&+\sqrt{\frac{1}{(d_2-\eta)}(\frac{2}{\eta}+\frac{n}{2d_2})}\Bigr[\sqrt{2}+\frac{(d_1+d_2)}{2\sqrt{(d_1-\epsilon)(d_2-\eta)}}\Bigr]\Bigr\},
\ea
$$
 there exists a constant $C(u_0, v_0)$ such that the unique nonnegative solution $(u,v)$ of the chemotaxis-growth system \eqref{para-para} satisfies
$$
\int_\Omega u^3(\cdot, t)+\int_\Omega u^2(\cdot, t)|\nabla v(\cdot, t)|^2+\int_\Omega u|\nabla v|^4+\int_\Omega |\nabla v(\cdot, t)|^6\leq C(u_0, v_0), \forall t\in (0, T_{\text{max}}).
$$
Thus, if $n\leq 5$, then the solution $(u,v)$ exists globally in time, i.e.,  $T_{\text{max}}=\infty$ and  $(u(\cdot,t), v(\cdot,t))$ is uniformly bounded in  $L^\infty(\Omega)\times W^{1,\infty}(\Omega)$ for all $t\in (0, \infty)$.
\end{lemma}
\begin{proof}Due to  $\mu>\mu_0$ given in \eqref{mu>mu0-4d}, we first know from Lemma  \ref{main-thm-3d} that
\be\label{u2+gradv4 bdd}
 3a\delta_1\int_\Omega u^2+2a\delta_2\int_\Omega u |\nabla v|^2+a\delta_3\int_\Omega |\nabla v|^4\leq K_{9}(u_0,v_0)
\ee
 for all $ t\in (0, T_{\text{max}})$. Moreover, due to \eqref{mu>mu0-4d0},  the fact that $\mu>\mu_0$ allows us to fix  $\epsilon, \eta, \epsilon_i, \delta_i$  satisfying  \eqref{delta2-delta3-re} and \eqref{mu-min-deltas-4d-general}. Indeed,  we first choose   $(\epsilon, \eta)=(\epsilon_0, \eta_0)$ so that $h$ is minimized, then we choose $
\epsilon_1$ to be the minimizer of
$$
\frac{\chi^2}{2\epsilon_1}+(\frac{2\alpha^2}{\eta}+\frac{n\alpha^2}{2d_2})\frac{\epsilon_1}{d_2-\eta},
$$
take  $\epsilon_3=1$ and then fix $\epsilon_2$ and $\epsilon_4$ so that \eqref{delta2-delta3-re} is satisfied. Then, based on \eqref{mu-min-deltas-4d-general}, \eqref{delta2-delta3-re} and \eqref{mu>mu0-4d0},  all $\delta_i$ can be chosen readily.

Upon such well chosen $\epsilon,\eta, \epsilon_i$ and $\delta_i$,  using \eqref{com-need-4d}, \eqref{bddary int n=4} and  arguing as Lemma \ref{main-thm-3d}, we can easily deduce a Gronwal inequality for the coupled quantity $z$ as defined by \eqref{couple-ineq-4d} of the form \eqref{gronwall-ineq-4d}. As a matter of fact, by \eqref{com-need-4d} and \eqref{bddary int n=4}, we have
\be\label{com-need-4d-growall}
\begin{array}{ll}
&z^\prime(t)+\beta z(t)+2[3(d_1-\epsilon)\delta_1-\epsilon_4\delta_2-\xi]\int_\Omega u |\nabla u|^2+(3\mu\delta_1-\frac{n\alpha^2}{18d_2}\delta_2)\int_\Omega u^4\\[0.25cm]
&\quad +2[3(d_2-\eta)\delta_4-(\epsilon_1+\epsilon_2)\delta_3-\sigma]\int_\Omega |\nabla v|^{2}|\nabla |\nabla v|^2|^2\\[0.25cm]
&\leq 3a\delta_1\int_\Omega u^2+2a\delta_2\int_\Omega u |\nabla v|^2+a\delta_3\int_\Omega |\nabla v|^4\\[0.25cm]
&\quad +\beta\delta_1\int_\Omega u^3+C_\xi\int_\Omega u^{\frac{3}{2}}+C_\sigma \int_\Omega  |\nabla v|^3.
\end{array}
\ee
Now, by fixing $\xi$ and $\sigma$ sufficiently small, we deduce from \eqref{u2+gradv4 bdd},  \eqref{com-need-4d-growall} and Holder inequality that
$$
z^\prime (t)+\beta z(t)\leq C(u_0,v_0),
$$
 giving the desired inequality stated in the lemma. In particular,  it assures  that $\|u(\cdot,t)\|_{L^3(\Omega)}$  is uniformly bounded. So the $L^{\frac{n}{2}+\epsilon}$-criterion \cite{BBTW15, Xiangpre} shows, if $n\leq 5$,  that $(u(\cdot,t), v(\cdot,t))$ is uniformly bounded in $L^\infty(\Omega)\times W^{1,\infty}(\Omega)$ for all $t\in (0, \infty)$.
\end{proof}

The algorithm for proving $L^2$ and $L^3$-boundedness of $u$  may be in principle  carried over to obtain $L^r$--boundedness of $u$ for $r=4,5,\cdots$ inductively by establishing a  Gronwal  inequality for  the coupled quantity (motivated by \cite{Win10} again)
\be\label{couple-ineq-nd}
z(t):=\sum_{k=0}^r \delta_i\int_\Omega u^k|\nabla v|^{2(r-k)},\quad \quad  t\in(0, T_{\text{max}})
\ee
of the form
$$
z^\prime(t)+\epsilon z(t)\leq C_\epsilon, \quad  t\in(0, T_{\text{max}})
$$
for some carefully chosen positive constants $\delta_i$ and small $\epsilon>0$ and perhaps large $C_\epsilon$ independent of $t$. In this process, as $r$ becomes large, we will have more terms to cope with and more boundary integrals will appear.  While simplification of the process is possible. In fact, for $r=2,3,\cdots$,  suppose that  $\mu>\mu_0^{(r)}(n, d_1,d_2,  \alpha, \chi)$ is the condition under which  $L^r$-boundedness of $u$ (indeed, uniform boundedness of $z(t)$ as in \eqref{couple-ineq-nd}) is guaranteed. Then, based on the procedures for $r=2$ and $r=3$,  if $\mu>\mu_0^{(r-1)}(n, d_1,d_2,  \alpha, \chi)$, the process for $r$  may be continued  with the assumption that $a=0$ and $\Omega$ is convex.

 Once this is done, the $L^{r}$-boundedness  of $u$ will be obtained. Further, by choosing  $r=\lfloor\frac{n}{2}\rfloor+1$,  the $L^{\frac{n}{2}+\epsilon}$-criterion in \cite{BBTW15, Xiangpre} yields the desired $L^\infty$-boundedness of $u$. Here, we state the following expected  general result that offers a quantitative description on when  logistic damping dominates over cheomtactic aggregation for \eqref{para-para} in $\Omega\subset \mathbb{R}^n$. While, we have to leave a rigorous examination for future study.
\begin{proposition}\label{main-thm-nd}  Let $\Omega\subset \mathbb{R}^ n$ be a bounded smooth domain,  $f$ satisfy  the logistic condition \eqref{log-con} and $d_1,d_2,\alpha,\beta>0$, $a\geq 0$ and $\chi\in\mathbb{R}$. Then, for any natural number $r\geq2$,  there exist  a function  $\theta_0^{(r)}(n,d_1,d_2)$ which tends to infinity as $d_1\rightarrow 0$ or $d_2\rightarrow 0$  (and is decreasing in $d_1, d_2$) and constant $C(u_0, v_0)$ such that, for any
$$
\mu>\theta_0^{(r)}(n,d_1,d_2)\alpha|\chi|,
$$
the nonnegative solution $(u,v)$ of the chemotaxis-growth system \eqref{para-para} satisfies
$$
\sum_{k=0}^r \int_\Omega u^k|\nabla v|^{2(r-k)}\leq C(u_0, v_0),\quad \quad  t\in(0, T_{\text{max}}).
$$
As a result, if $n\leq 2r-1$, then  $(u,v)$ exists globally in time, i.e.,  $T_{\text{max}}=\infty$ and  $(u(\cdot,t), v(\cdot,t))$ is uniformly bounded in $L^\infty(\Omega)\times W^{1,\infty}(\Omega)$ for all $t\in (0, \infty)$.
\end{proposition}

\subsection{The special case that $d_1=d_2$ and  $\Omega\subset \mathbb{R}^n$ is convex}

In the special case that $\tau=1$ in \eqref{minimal-model}, $\chi>0$ (positive chemotaxis)   and $\Omega\subset \mathbb{R}^n$ is convex, the explicit lower bound
$$
\mu>\frac{n}{4}\chi
$$
ensuring global boundedness to the solution of \eqref{minimal-model} has  been elucidated  in \cite{Win10} by establishing a parabolic inequality for a combination of $u$ and $|\nabla v|^2$. Here, in this subsection, for the full-parameter chemotaxis model  \eqref{para-para}, we will  write down all the details for convenience and completeness.
\begin{lemma}\label{main-thm-nd-convex}  Let $\Omega\subset \mathbb{R}^ n$ be a bounded smooth convex domain,  $f$ satisfy  the logistic condition \eqref{log-con} and $d_1,d_2,\alpha,\beta, \chi>0$, $a\geq 0$. Assume that
$$
d_1=d_2, \quad \quad \mu>\frac{n}{4d_1}\alpha\chi,
$$
the unique nonnegative solution $(u,v)$ of the chemotaxis-growth system \eqref{para-para}   exists globally in time, i.e.,  $T_{\text{max}}=\infty$ and  is bounded in the following way:
$$
 u\leq \max\Bigr\{\max_{\bar{\Omega}\times [0,1]}u, \frac{1}{2\alpha}\max\Bigr\{ \max_{\bar{\Omega}}(2\alpha u(\cdot,1)+\chi|\nabla v(\cdot,1)|^2), \frac{a\alpha}{\beta}+\frac{\alpha \beta}{\mu-\frac{n\alpha}{4d_1}\chi}\Bigr\} \Bigr\}
$$
and
$$
v\leq \max\Bigr\{\max_{\bar{\Omega}\times [0,1]}v, \frac{1}{2\beta}\max\Bigr\{ \max_{\bar{\Omega}}(2\alpha u(\cdot,1)+\chi|\nabla v(\cdot,1)|^2), \frac{a\alpha}{\beta}+\frac{\alpha \beta}{\mu-\frac{n\alpha}{4d_1}\chi}\Bigr\} \Bigr\}
$$
on $\bar{\Omega}\times [0, \infty)$. In addition, for any $\epsilon>0$,
$$
|\nabla v|\leq  \frac{1}{\chi}\max\Bigr\{ \max_{\bar{\Omega}}(2\alpha u(\cdot,\epsilon)+\chi|\nabla v(\cdot,\epsilon)|^2), \frac{a\alpha}{\beta}+\frac{\alpha \beta}{\mu-\frac{n\alpha}{4d_1}\chi}\Bigr\} \text{ on  } \bar{\Omega}\times [\epsilon,\infty).
$$
Moreover, if $\|\nabla v_0\|_{L^\infty(\Omega)}<\infty$, then $\epsilon$ can be chosen to be zero.
\end{lemma}

\begin{proof}Taking  the gradient first and then multiplying  it by $\bigtriangledown v$ in the $v$-equation of (\ref{para-para}) and then using \eqref{point-wise-id} and the fact that $d_1=d_2$, one derives
\be\label{dgradvdt-c}
(|\bigtriangledown v|^2)_t
=d_1\bigtriangleup|\bigtriangledown v|^2-2d_1|D^2v|^2-2\beta |\bigtriangledown v|^2+2 \alpha\bigtriangledown u \bigtriangledown v.
\ee
 Multiplying the $u$-equation of \eqref{para-para} by $2\alpha $ and the equation \eqref{dgradvdt-c} by $\chi$ yields
$$
(2\alpha u+\chi|\nabla v|^2)_t=d_1\Delta (2\alpha u+\chi|\nabla v|^2)-2d_1\chi|D^2v|^2-2\beta\chi|\nabla v|^2-2\alpha \chi u\Delta v+2\alpha f(u).
$$
Notice from Cauchy-Schwarz inequality and  \eqref{C-S-ineq-c} that
$$
-2\alpha \chi u\Delta v\leq\frac{n\chi\alpha^2}{2d_1}u^2+ \frac{2d_1}{n}\chi|\Delta v|^2\leq \frac{n\chi\alpha^2}{2d_1}u^2+2d_1\chi|D^2v|^2,
$$
we then deduce from the logistic condition $f(u)\leq a-\mu u^2$ that
\be\label{para-max-prin}
(2\alpha u+\chi|\nabla v|^2)_t\leq d_1\Delta (2\alpha u+\chi|\nabla v|^2)-2\beta\chi|\nabla v|^2+2a\alpha-2\alpha (\mu-\frac{n\alpha}{4d_1}\chi)u^2.
\ee
Recall that, for a  convex domain $\Omega$, we have the well-known  fact that $\frac{\partial}{\partial \nu}(|\bigtriangledown v|^2)\leq 0$ for any function satisfying $\frac{\partial v}{\partial \nu}=0$  on $\partial \Omega$ ; see Matano \cite[Lemma 5.3]{MA79} or \cite{Win10} or \cite{DGG98}.  Consequently, by the homogeneous boundary conditions for $u$ and $v$, we get
\be\label{para-max-prin-w}
\frac{\partial }{\partial \nu} w:=\frac{\partial }{\partial \nu}(2\alpha u+\chi|\nabla v|^2)=2\alpha\frac{\partial u}{\partial \nu}+2\chi\frac{\partial}{\partial \nu}|\nabla v|^2=2\chi\frac{\partial}{\partial \nu}|\nabla v|^2\leq 0\text{ on } \partial \Omega.
\ee
Now, since  $\mu>\frac{n\alpha }{4d_1}\alpha \chi$, using elementary calculations we find that
\be\label{para-max-prinw}
w_t- d_1\Delta w+2\beta w\leq M,
\ee
where
$$
 M=\max\{2a\alpha-2\alpha (\mu-\frac{n\alpha}{4d_1}\chi)u^2+4\alpha\beta u|u\geq 0\}=2a\alpha+\frac{2\alpha\beta^2}{\mu-\frac{n\alpha}{4d_1}\chi} <\infty.
$$
Recall that from Lemma \ref{local-in-time} that $(u,v)\in C^{2,1}(\bar{\Omega}\times (0, T_{\text{max}}))$. So, we will perform a small time shift and treat any positive time as the "initial time".  Then, by \eqref{para-max-prinw} and \eqref{para-max-prin-w}, we conclude from the  maximum principle and the Hopf boundary point lemma, for any $\epsilon\in (0, T_{\text{max}})$,  that
$$
w=2\alpha u+\chi|\nabla v|^2\leq \max\Bigr\{ \max_{\bar{\Omega}}(2\alpha u(\cdot,\epsilon)+\chi|\nabla v(\cdot,\epsilon)|^2), \frac{ M}{2\beta}\Bigr\}\text{ on } \bar{\Omega}\times [\epsilon, T_{\text{max}}).
$$
This gives the uniform  boundedness for the coupled quantity $2\alpha u+\chi|\nabla v|^2$. The positiveness of $\chi$ then  directly implies the $L^\infty$-boundedness of $u$ and $\nabla v$:
$$u
\leq \max\Bigr\{\max_{\bar{\Omega}\times [0,\epsilon]}u, \frac{1}{2\alpha}, \max\Bigr\{ \max_{\bar{\Omega}}(2\alpha u(\cdot, \epsilon)+\chi|\nabla v(\cdot,\epsilon)|^2), \frac{ M}{2\beta}\Bigr\} \Bigr\}
$$
on  $\bar{\Omega}\times [\epsilon, T_{\text{max}})$ and
$$
|\nabla v|
\leq \frac{1}{\chi}\max\Bigr\{ \max_{\bar{\Omega}}(2\alpha u(\cdot,\epsilon)+\chi|\nabla v(\cdot,\epsilon)|^2), \frac{ M}{2\beta}\Bigr\}\text{ on } \bar{\Omega}\times [\epsilon, T_{\text{max}}).
$$
These combined with the blow-up criterion \eqref{T_m-cri}  of Lemma \ref{local-in-time} show  that $(u,v)$ exists globally in time, i.e., $T_{\text{max}}=\infty$.

Finally, an application of maximum principle to the $v$-equation of \eqref{para-para} on the region $\bar{\Omega}\times [1,\infty)$ gives the bound for $v$. Indeed, $|\nabla v|$ is bounded if $|\nabla v_0|$ is.
\end{proof}

\section{Proof of the large time behavior for the KS model \eqref{min-std}}
In this section, we show the proof of Theorem \ref{large time}, which  relies on finding  so-called Lyapounov  functionals. Here, we will present all the  necessary details for the clarity of obtaining the explicit convergence rates.
 \begin{proof} We modified the functional in \cite{HZ16} as
 \be\label{H-def}
 H(t)=\int_\Omega \Bigr(u-\frac{\kappa}{\mu}-\frac{\kappa}{\mu}\ln (\frac{\mu}{\kappa}u)\Bigr)+\delta\int_\Omega (v-\frac{\alpha\kappa}{\beta\mu})^2, \quad \delta=\frac{\kappa\chi^2}{8d_1d_2\mu}.
 \ee
 Differentiating $H$, using  \eqref{min-std} and integrating by parts, we deduce from Cauchy-Schwarz inequality that

\begin{equation}\label{diff-H2}\begin{split}
\frac{d}{dt} H(t) &=  \int_\Omega\frac{u-\frac{\kappa}{\mu}}{u}u_t+2\delta\int_\Omega (v-\frac{\alpha\kappa}{\beta\mu})v_t\\
&\quad +2\delta \int_\Omega (v-\frac{\alpha\kappa}{\beta\mu})(d_2\Delta  v  -\beta v+\alpha u)\\
&=\int_\Omega\frac{u-\frac{\kappa}{\mu}}{u}\Bigr(\nabla \cdot (d_1\nabla u-\chi u\nabla v)+u(\kappa -\mu u)\Bigr)\\
&\quad -2d_2\delta \int_\Omega |\nabla v|^2+2\delta\int_\Omega(v-\frac{\alpha\kappa}{\beta\mu})(-\beta v+\alpha u) \\
&=-\frac{\kappa}{\mu}d_1\int_\Omega \frac{|\nabla u|^2}{u^2}+\frac{\kappa}{\mu}\chi \int_\Omega \frac{\nabla u}{u}\cdot\nabla v-\mu\int_\Omega (u-\frac{\kappa}{\mu})^2\\
&-2d_2\delta \int_\Omega |\nabla v|^2-2\beta\delta\int_\Omega (v-\frac{\alpha\kappa}{\beta\mu})^2+2\alpha\delta\int_\Omega (u-\frac{\alpha\kappa}{\beta\mu})(v-\frac{\alpha\kappa}{\beta\mu})\\
&\leq -2d_2(\delta-\frac{\kappa\chi^2}{8d_1d_2\mu}) \int_\Omega |\nabla v|^2-(\mu-2\alpha\delta\epsilon)\int_\Omega (u-\frac{\kappa}{\mu})^2\\
&\quad \quad \quad \quad \quad \quad \ \ \ \ \ -2\delta(\beta-\frac{\alpha}{4\epsilon})\int_\Omega (v-\frac{\alpha\kappa}{\beta\mu})^2\\
&=-(\mu-2\alpha\delta\epsilon)\int_\Omega (u-\frac{\kappa}{\mu})^2-2\delta(\beta-\frac{\alpha}{4\epsilon})\int_\Omega (v-\frac{\alpha\kappa}{\beta\mu})^2
\end{split}\end{equation}
for any $\epsilon>0$. Now, we wish to minimize $\mu$ by choosing
\be\label{e-con}
\left\{ \begin{array}{ll}
 \mu-2\alpha\delta\epsilon>0,\\[0.2cm]
 \beta-\frac{\alpha}{4\epsilon}>0.
 \end{array}\right.\Longleftrightarrow \left\{ \begin{array}{ll}
 \mu>\frac{\kappa\alpha\chi^2}{4d_1d_2\mu}\epsilon,\\[0.2cm]
 \frac{\alpha}{4\beta}<\epsilon<\frac{4d_1d_2\mu^2}{\alpha\kappa \chi^2}.
 \end{array}\right. \Longleftrightarrow\mu^2>\frac{\kappa\alpha^2\chi^2}{16d_1d_2\beta}.
\ee
The latter is guaranteed by our assumption \eqref{mu-con}.
Now, for fixed $\epsilon$ obeying \eqref{e-con}, we set
\be\label{eta-con}
\eta=\min\Bigr\{(\mu-2\alpha\delta\epsilon),2\delta(\beta-\frac{\alpha}{4\epsilon}) \Bigr\}=\min\Bigr\{(\mu-\frac{\alpha \kappa\chi^2}{4d_1d_2\mu}\epsilon),\frac{\kappa\chi^2}{4d_1d_2\mu}(\beta-\frac{\alpha}{4\epsilon})\Bigr\},
\ee
and then we infer from \eqref{diff-H2}, \eqref{e-con} and \eqref{eta-con} that
\begin{equation}\label{diff-H}
\frac{d}{dt} H(t) \leq -\eta\Big(\int_\Omega (u-\frac{\kappa}{\mu})^2+\int_\Omega (v-\frac{\alpha\kappa}{\beta\mu})^2\Bigr).
 \end{equation}
 Since $H(t)\geq 0$, an integration of \eqref{diff-H} from any $t_0\geq 0$ to $t$ yields
 $$
\eta \int_{t_0}^t\Big(\int_\Omega (u-\frac{\kappa}{\mu})^2+\int_\Omega (v-\frac{\alpha\kappa}{\beta\mu})^2\Bigr)\leq \frac{H(t_0)}{\eta}
 $$
 giving trivially
 \be\label{int-bdd}
 \int_{t_0}^\infty\Big(\int_\Omega (u-\frac{\kappa}{\mu})^2+\int_\Omega (v-\frac{\alpha\kappa}{\beta\mu})^2\Bigr)\leq \frac{H(t_0)}{\eta}.
 \ee
 Thanks to Theorem \ref{main thm}, the condition $\mu>\mu_0$ ensures that $(u,v)$ is globally bounded and classical. Then from the parabolic regularity, we see that $\int_\Omega (u-\frac{\kappa}{\mu})^2+\int_\Omega (v-\frac{\alpha\kappa}{\beta\mu})^2$ is uniformly bounded and uniformly continuous in $t$. This allows one to deduce from
 \eqref{int-bdd} that
 \be\label{int-lt}
  \int_\Omega (u-\frac{\kappa}{\mu})^2+\int_\Omega (v-\frac{\alpha\kappa}{\beta\mu})^2 \rightarrow 0, \quad \text{ as } t\rightarrow \infty.
 \ee
 Since $u,v$ are smooth and bounded, the standard parabolic regularity for parabolic  equations (cf. \cite{La})  shows the existence of $\sigma\in(0, 1)$ and $C$ such that
 \be\label{reg-uv}
 \|u\|_{C^{2+\sigma, 1+\frac{\sigma}{2}}(\bar{\Omega}\times[t,t+1])}+\|v\|_{C^{2+\sigma, 1+\frac{\sigma}{2}}(\bar{\Omega}\times[t,t+1])}\leq C, \quad \forall t\geq 1.
 \ee
 Hence, in view of the Gagliardo-Nirenberg inequality, \eqref{int-lt} and \eqref{reg-uv}, we obtain
\be\label{u-infty-0}\ba{ll}
\|u(\cdot,t)-\frac{\kappa}{\mu}\|_{L^\infty(\Omega)}&\leq C_{GN} \|u (\cdot,t)-\frac{\kappa}{\mu}\|_{W^{1,\infty}(\Omega)}^{\frac{n}{n+2}}\|u -\frac{\kappa}{\mu}\|
_{L^2(\Omega)}^{\frac{2}{n+2}}\\[0.25cm]
&\leq C\|u(\cdot,t)-\frac{\kappa}{\mu}\|
_{L^2(\Omega)}^{\frac{2}{n+2}} \rightarrow 0, \text{ as } t\rightarrow \infty.
\ea
\ee
In the same way, we get
\be\label{v-infty-0}
\|v(\cdot,t)-\frac{\alpha\kappa}{\beta\mu}\|_{L^\infty(\Omega)}
\leq C\|v(\cdot,t)-\frac{\alpha\kappa}{\beta\mu}\|
_{L^2(\Omega)}^{\frac{2}{n+2}} \rightarrow 0, \text{ as } t\rightarrow \infty.
\ee
 Based on the definition of $H$ in \eqref{H-def},  \eqref{u-infty-0} and \eqref{diff-H}, we calculate via the L'Hospital rule that
$$
\lim_{u\rightarrow \frac{\kappa}{\mu}}\frac{u-\frac{\kappa}{\mu}-\frac{\kappa}{\mu}\ln (\frac{\mu}{\kappa}u)}{(u-\frac{\kappa}{\mu})^2}=\frac{\mu}{2 \kappa}.
$$
This together with \eqref{u-infty-0} allows one to find $t_1\geq 0$ such that
$$
\frac{\mu}{4\kappa}(u-\frac{\kappa}{\mu})^2\leq u-\frac{\kappa}{\mu}-\frac{\kappa}{\mu}\ln (\frac{\mu}{\kappa}u)\leq \frac{\mu}{\kappa}(u-\frac{\kappa}{\mu})^2, \quad \quad t\geq t_1,
$$
and then \eqref{H-def} entails
\be\label{h-com}
 \min\{\frac{\mu}{4\kappa},\delta\} \Big(\int_\Omega (u-\frac{\kappa}{\mu})^2+\int_\Omega (v-\frac{\alpha\kappa}{\beta\mu})^2\Bigr)\leq H(t), \quad \quad t\geq t_1
\ee
and
\be\label{h-com2}
 H(t)\leq \max\{\frac{\mu}{\kappa},\delta\} \Big(\int_\Omega (u-\frac{\kappa}{\mu})^2+\int_\Omega (v-\frac{\alpha\kappa}{\beta\mu})^2\Bigr), \quad \quad t\geq t_1.
\ee
Combining \eqref{diff-H} and \eqref{h-com2}, we obtain an ordinary differential inequality for $H$:
$$
\frac{d}{dt} H(t)\leq -\frac{\eta}{\max\{\frac{\mu}{\kappa},\delta\}}H(t), \quad t\geq t_1,
$$
directly yielding
$$
H(t)\leq H(t_1)e^{-\frac{\eta}{\max\{\frac{\mu}{\kappa},\delta\}}(t-t_1)},  \quad t\geq t_1.
$$
This in conjunction with \eqref{h-com2} concludes that
$$
\int_\Omega (u-\frac{\kappa}{\mu})^2+\int_\Omega (v-\frac{\alpha\kappa}{\beta\mu})^2\leq \frac{H(t_1)}{ \min\{\frac{\mu}{4\kappa},\delta\}}e^{-\frac{\eta}{\max\{\frac{\mu}{\kappa},\delta\}}(t-t_1)},  \quad t\geq t_1.
$$
With this decay estimate at hand, we then derive from \eqref{u-infty-0} and \eqref{v-infty-0} that there exists a large constant $C>0$ such that
$$
\|u(\cdot, t)-\frac{\kappa}{\mu}\|_{L^\infty(\Omega)}+\|v(\cdot, t)-\frac{\alpha\kappa}{\beta\mu}\ \|_{L^\infty(\Omega)}\leq Ce^{-\frac{\eta}{(n+2)\max\{\frac{\mu}{\kappa},\delta\}}(t-t_1)},  \quad t\geq t_1.
$$
Substituting into the definitions of $\delta$ in \eqref{H-def}, $\eta$ in \eqref{eta-con} and taking $\epsilon=\epsilon_0$ in \eqref{e-con}, we obtain the exponential decay estimate \eqref{u+v-conv-i} by choosing large constant $C$.

In the case of $\kappa=0$, the uniform boundedness and global existence of solution does not affect as long as $\mu>\mu_0$. We integrate the first equation in \eqref{min-std}  and use Holder inequality to obtain
$$
\frac{d}{dt}\int_\Omega u=-\mu\int_\Omega u^2\leq -\mu |\Omega|^{-1 }\Big(\int_\Omega u)^{2}, \quad t>0,
$$
which entails
\be\label{int-u-0-a=0}
\int_\Omega u\leq \Bigr[(\int_\Omega u_0)^{-1}+\mu|\Omega|^{-1}t\Bigr]^{-1}\leq \frac{c_1}{t+1}, \quad t> 0.
\ee
Hence, the Gagliardo-Nirenberg inequality coupled with the boundedness of $u$ shows that
\be\label{u-infty-0-a=0}\ba{ll}
\|u(\cdot,t)\|_{L^\infty(\Omega)}
&\leq C_{GN}\|u(\cdot,t)\|_{W^{1,\infty}(\Omega)}^{\frac{n}{n+1}}   \|u(\cdot,t) \|
_{L^1(\Omega)}^{\frac{1}{n+1}}\\[0.25cm]
&\leq C \Bigr[(\int_\Omega u_0)^{-1}+\mu|\Omega|^{-1}t\Bigr]^{-\frac{1}{n+1}}, \quad t> 0.
\ea
\ee
An integration of the second equation in \eqref{min-std} shows
\be\label{v-gron}
\frac{d}{dt}\int_\Omega v=-\beta \int_\Omega v+\alpha \int_\Omega u \leq-\beta \int_\Omega v+\frac{c_2}{t+1}.
\ee
Solving this Gronwall inequality shows
$$
\int_\Omega v \leq \|v_0\|_{L^1} e^{-\beta t} +c_2\frac{\int_0^t\frac{e^{\beta s}}{s+1}}{e^{\beta t}}\leq \|v_0\|_{L^1} e^{-\beta t} +\frac{c_3}{t+1}\leq \frac{c_4}{t+1},
$$
where we used the fact that
$$
\lim_{t\rightarrow \infty}\frac{(t+1)\int_0^t\frac{e^{\beta s}}{s+1}}{e^{\beta t}}=\frac{1}{\beta}<\infty, \quad \lim_{t\rightarrow \infty}(t+1)e^{-\beta t}=0.
$$
Then we conclude from \eqref{u-infty-0-a=0} with $u$ replaced by $v$ that, $t>0$,
\be\label{v-infty-0-a=0}
\|v(\cdot,t)\|_{L^\infty(\Omega)}\leq \frac{c_4}{(t+1)^{n+1}}
\ee
In the case of $\kappa<0$, we integrate the first equation in \eqref{min-std} to get
$$
\frac{d}{dt}\int_\Omega u=\kappa\int_\Omega u-\mu\int_\Omega u^{2}\leq \kappa \int_\Omega u, \quad t>0,
$$
and thus
\be\label{int-u-0-a<0}
\int_\Omega u\leq e^{\kappa t } \int_\Omega u_0, \quad t> 0.
\ee
Then the GN inequality \eqref{u-infty-0-a=0} implies
\be\label{u-infty-0-a<0}
\|u(\cdot,t)\|_{L^\infty(\Omega)}
\leq c_5e^{\frac{\kappa }{n+1}t },   \quad t> 0.
\ee
Combining \eqref{int-u-0-a<0} and \eqref{v-gron}, we have
$$
\|v(\cdot,t)\|_{L^1}\leq \left\{ \begin{array}{ll}
\|v_0\|_{L^1}e^{-\beta t}+c_6t e^{-\beta t}, \text{ if } \beta=-\kappa, \\[0.2cm]
\|v_0\|_{L^1}e^{-\beta t}+c_7\frac{e^{\kappa t}- e^{-\beta t}}{\beta+\kappa}, \text{ if } \beta\neq -\kappa. \end{array}\right.\leq c_9e^{-\frac{1}{2}\min\{\beta, -\kappa\}t}
$$
With this decay estimate  at hand, the GN inequality \eqref{u-infty-0-a=0} gives rise to
\be\label{v-infty-0-a<0}
\|v(\cdot,t)\|_{L^\infty(\Omega)}\leq c_{10} e^{-\frac{1}{2(n+1)}\min\{\beta, -\kappa\}t}, \forall t\geq 0.
\ee
Extracting  the essential ingredients of the estimates \eqref{u-infty-0-a=0}, \eqref{v-infty-0-a=0}, \eqref{u-infty-0-a<0} and  \eqref{v-infty-0-a<0}, we readily conclude the decay estimates  \eqref{u+v-conv-i},  \eqref{u+v-conv-ii} and \eqref{u+v-conv-iii}.
 \end{proof}

\textbf{Acknowledgments}   This work  was  funded by the National Natural Science Foundation of China (No. 11601516 and   11571364), the Fundamental Research Funds for the Central Universities and the Research Funds of Renmin University of China (No. 15XNLF10).
.

\end{document}